\documentclass[12pt,a4paper,oneside,reqno]{amsart}
\usepackage[a4paper,hmargin={2.4cm,2.4cm},vmargin={3cm,3cm}]{geometry}
\usepackage[utf8]{inputenc}
\usepackage{amsmath}
\usepackage{amsfonts}
\usepackage{amssymb}
\usepackage{amsthm}
\usepackage{enumerate}
\usepackage{color}
\usepackage{url}
\usepackage{graphicx}
\usepackage{hyperref}
\usepackage{mathrsfs}

\numberwithin{equation}{section}
\theoremstyle{plain}
\newtheorem*{theorem*}{Theorem}
\newtheorem{theorem}{Theorem}[section]
\newtheorem{lemma}[theorem]{Lemma}
\newtheorem{proposition}[theorem]{Proposition}
\newtheorem{corollary}[theorem]{Corollary}

\theoremstyle{remark}
\newtheorem{remark}[theorem]{Remark}
\newtheorem*{remark*}{Remarks}

\theoremstyle{definition}

\newcommand{\thmref}[1]{Theorem~\ref{#1}}
\newcommand{\secref}[1]{Section~\ref{#1}}
\newcommand{\lemref}[1]{Lemma~\ref{#1}}
\newcommand{\propref}[1]{Proposition~\ref{#1}}
\newcommand{\corref}[1]{Corollary~\ref{#1}}
\newcommand{\remref}[1]{Remark~\ref{#1}}

\newcommand{\EE}{\mathbb{E}}

\newcommand{\HH}{\mathbb{H}}
\newcommand{\NN}{\mathbb{N}}
\newcommand{\PP}{\mathbb{P}}
\newcommand{\RR}{\mathbb{R}}

\newcommand{\ZZ}{\mathbb{Z}}

\newcommand{\cA}{\mathcal{A}}

\newcommand{\cC}{\mathcal{C}}
\newcommand{\cD}{\mathcal{D}}
\newcommand{\cE}{\mathcal{E}}
\newcommand{\cF}{\mathcal{F}}
\newcommand{\cG}{\mathcal{G}}

\newcommand{\cK}{\mathcal{K}}

\newcommand{\cM}{\mathcal{M}}
\newcommand{\cN}{\mathcal{N}}
\newcommand{\cP}{\mathcal{P}}

\newcommand{\cS}{\mathcal{S}}
\newcommand{\cT}{\mathcal{T}}

\newcommand{\cV}{\mathcal{V}}
\newcommand{\cW}{\mathcal{W}}

\newcommand{\Ind}[1]{\mathbf{1}_{\{#1\}}}
\newcommand{\tx}{{T_{\mathrm{mix}}}}
\newcommand{\Var}{\operatorname{Var}}
\newcommand{\Cov}{\operatorname{Cov}}

\newcommand{\floor}[1]{\left\lfloor #1 \right\rfloor}
\newcommand{\ceil}[1]{\left\lceil #1 \right\rceil}
\newcommand{\overE}{{\overline{\mathbb E}}}
\newcommand{\err}{\operatorname{err}}

\begin{document}

\title{Aging of the Metropolis dynamics on the Random Energy Model}
\author{Jiří Černý}
\address{\tiny Jiří Černý\\University of Vienna\\Faculty of Mathematics\\
Oskar-Morgenstern-Platz 1\\A-1090 Wien\\Austria}
\email{\href{mailto:jiri.cerny@univie.ac.at}
{jiri.cerny@univie.ac.at}}
\author{Tobias Wassmer}
\address{\tiny Tobias Wassmer\\University of Vienna\\Faculty of Mathematics\\
Oskar-Morgenstern-Platz 1\\A-1090 Wien\\Austria}
\email{\href{mailto:tobias.wassmer@univie.ac.at}
{tobias.wassmer@univie.ac.at}}

\begin{abstract}
  We study the Metropolis dynamics of the simplest mean-field spin glass
  model, the Random Energy Model. We show that this dynamics exhibits
  aging by showing that the properly rescaled time change process between
  the Metropolis dynamics and a suitably chosen `fast' Markov chain
  converges in distribution to a stable subordinator. The rescaling might
  depend on the realization of the environment, but we show that its
  exponential growth rate is deterministic.
\end{abstract}

\date{\today}

\maketitle


\section{Introduction} 

This paper studies the out-of-equilibrium behavior of the Metropolis
dynamics on the Random Energy Model (REM). Our main goal is to answer one
of the remaining important open questions in the field, namely whether
this dynamics exhibits aging, and, if yes, whether its aging behavior
admits the usual description in terms of stable Lévy processes.

Aging is one of the main features appearing in the long-time behavior of
complex disordered systems (see e.g.~\cite{BCKM97} for a review). It was for
the first time observed experimentally in the anomalous relaxation patterns
of the residual magnetization of spin glasses (e.g.~\cite{LSNB83,Cha84}).
One of the most influential steps in the theoretical modeling of the aging
phenomenon is the introduction of the so-called trap models by Bouchaud
\cite{B92} and Bouchaud and Dean \cite{BD95}. These models, while being
sufficiently simple to allow analytical treatment, reproduce the
characteristic power law decay seen experimentally.

Since then a considerable effort has been made in putting the predictions
obtained from the trap models to a solid basis, that is to derive
these predictions from an underlying spin-glass dynamics. The first
attempt in this direction was made in \cite{BBG02,BBG03a,BBG03b} where it
was shown that, for a very particular Glauber-type dynamics, at time
scales very close to the equilibration, a well chosen two-point
correlation function converges to that given by Bouchaud's trap model.

With the paper \cite{bAC08}, where the same type of dynamics was studied
in a more general framework and on a broader range of time scales, it
emerged that aging establishes itself by the fact that scaling limits of
certain additive functionals of Markov chains are stable Lévy processes,
and that the convergence of the two-point correlation functions is just a
manifestation of the classical arcsine law for stable subordinators.

The Glauber-type dynamics used in those papers, sometimes called random
hopping time (RHT) dynamics, is however rather simple and is often
considered as `non-realistic', mainly because its transition rates do not
take into account the energy of the target state. Its advantage is that it
can be expressed as a time change of a simple random walk on the
configuration space of the spin glass, which allows for a certain
decoupling of the randomness of the dynamics from the randomness of the
Hamiltonian of the spin glass, making its rigorous studies more tractable.

For more realistic Glauber-type dynamics of spin glasses, like the
so-called Bouchaud's asymmetric dynamics or the Metropolis dynamics,
such decoupling is not possible. As a consequence, these dynamics are
far less understood.

Recently, some progress has been achieved in the context of the simplest
mean-field spin glass model, the REM. First, in \cite{MM12}, the Bouchaud's
asymmetric dynamics have been considered in the regime
where the asymmetry parameter tends to zero with the size of the system.
Building on the techniques started in \cite{Mou11}, this papers confirms
the predictions of Bouchaud's trap model in this regime.
Second, the Metropolis dynamics have been studied in \cite{Gay14},
for a truncated version of the REM, using the techniques developed for
the symmetric dynamics in \cite{Gay12,Gay10}, again confirming Bouchaud's
predictions.

The weak asymmetry assumption of \cite{MM12} and the truncation of
\cite{Gay14} have both the same purpose. They aim at overcoming some
specific features of the asymmetry and recovering certain features of
symmetric dynamics. Our aim in this work is to get rid of this
simplifications and treat the non-modified REM with the usual Metropolis
dynamics.

Let us also mention that Bouchaud's asymmetric dynamics (and implicitly
  the Metropolis one) is rather well understood in the context of trap
models on $\mathbb Z^d$, see \cite{BC11,Cer11,GS13}, where it is possible
to exploit the connections to the random conductance model with unbounded
conductances, \cite{BD10}. Finally, the Metropolis dynamics on the complete
graph was considered in \cite{Gay12}.

\medskip

Before stating our main result, let us briefly recall the general scheme for
proving aging in terms of convergence to stable Lévy processes. The actual
spin glass dynamics, $X=(X_t)_{t\geq0}$, which is reversible with respect to
the Gibbs measure of the Hamiltonian, is compared to another Markov chain
$Y=(Y_t)_{t\geq0}$ on the same space, which is an `accelerated' version of
$X$ and whose stationary measure is uniform.
The process $Y$ is typically easier to be understood, e.g.~it is a simple
random walk for the RHT dynamics, and the original process $X$ can be
written as its time change,
\begin{equation}
  \label{e:timechange}
  X(t)=Y(S^{-1}(t)),
\end{equation}
for the right continuous inverse $S^{-1}$ of a certain additive
functional $S$ of the Markov chain $Y$, called the `clock process'.
The aim is then to show convergence of the properly rescaled clock
process $S$ to an increasing stable Lévy process, that is to a stable
subordinator.

\medskip

We now state our main result. We consider the unmodified REM, as introduced in
\cite{D80,D81}. The state space of this model is the
$N$-dimensional hypercube $\HH_N=\{-1,1\}^N$, and its Hamiltonian is a
collection $(E_x)_{x\in\HH_N}$ of i.i.d.~standard Gaussian random variables
defined on some probability space $(\Omega, \cF,\PP)$. The non-normalized
Gibbs measure $\tau_x = e^{\beta\sqrt{N}E_x}$ at inverse temperature $\beta>0$
gives the equilibrium distribution of the system.

The Metropolis dynamics on the REM is the continuous-time Markov chain
$X=(X_t)_{t\geq0}$ on $\HH_N$ with transition rates
\begin{equation} \label{def:X}
  r_{xy}= \left(1\wedge \frac{\tau_y}{\tau_x}\right)\Ind{x\sim y},
  \qquad x,y\in \mathbb H_N.
\end{equation}
Here, $x\sim y$ means that $x$ and $y$ are neighbors on $\mathbb H_N$,
that is they differ in exactly one coordinate.

As explained above, we will compare the Metropolis chain $X$ with another
`fast' Markov chain $Y = (Y_t)_{t\ge 0}$ with transition rates
\begin{equation}\label{def:Y}
  q_{xy}= \frac{\tau_x\wedge\tau_y}
  {1\wedge\tau_x}\Ind{x\sim y},\qquad x,y\in \mathbb H_N.
\end{equation}
It can be easily checked using the detailed balance conditions that $Y$ is
reversible and that its
equilibrium distribution is
\begin{equation*}
  \nu_x=\frac{1\wedge\tau_x}{Z_N}, \qquad x\in \mathbb H_N,
\end{equation*}
where $Z_N= \sum_{x\in\HH_N}(1\wedge\tau_x)$.  Finally, since
$r_{xy} = (1\vee\tau_x)^{-1} q_{xy}$, $X$ can be written as a time change
of $Y$ as in \eqref{e:timechange} with the clock process $S$ being given
by
\begin{equation}
  \label{e:clock}
  S(t) = \int_0^t (1\vee\tau_{Y_s})ds.
\end{equation}
For the rest of the paper we only deal with the process $Y$ and the clock
process $S$, the actual Metropolis dynamics $X$ does not appear anymore after
this point.
For a fixed environment $\tau=(\tau_x)_{x\in\HH_N}$, let
$P^{\tau}_{\nu}$ denote the law of the process $Y$ started from its
stationary distribution $\nu $, and
let $D([0,T],\RR)$ be the space of $\RR$-valued cadlag functions on $[0,T]$.
We denote by $\beta_c = \sqrt {2 \log 2}$ the (static) critical
temperature of the REM.
Our main result is the following.

\begin{theorem}
  \label{thm:main}
  Let $\alpha\in(0,1)$ and $\beta>0$ be such that
  \begin{equation}\label{cond:alphabeta}
    \frac{1}{2} < \frac{\alpha^2\beta^2}{\beta_c^2} < 1,
  \end{equation}
  and define
  \begin{equation}
    \label{def:gn}
    g_N=e^{\alpha \beta^2 N}(\alpha\beta\sqrt{2\pi N})^{-\frac{1}{\alpha}}.
  \end{equation}
  Then there are random variables $R_N$ which depend on the environment
  $(E_x)_{x\in\HH_N}$ only, such that for every $T>0$ the rescaled clock processes
  \begin{equation*}
    S_N(t) = g_N^{-1}S(tR_N),\qquad t\in [0,T],
  \end{equation*}
  converge in $\PP$-probability as $N\to\infty$, in
  $P^{\tau}_{\nu}$-distribution on the space $D([0,T],\RR)$ equipped with
  the Skorohod $M_1$-topology, to an $\alpha$-stable subordinator
  $V_{\alpha}$. The random variables $R_N$ satisfy
  \begin{equation}
    \label{e:Rass}
    \lim_{N\to\infty}\frac{\log R_N}{N}  = \frac{\alpha^2 \beta^2}{2},
    \quad \PP\text{-a.s.}
  \end{equation}
\end{theorem}

\noindent Let us make a few remarks on this result.

1. The result of Theorem~\ref{thm:main} confirms that the predictions of
Bouchaud's trap model hold for the Metropolis dynamics on the REM, at
least at the level of scaling limits of clock processes. It also compares
directly to the results obtained for the symmetric (RHT) dynamics in
\cite{bAC08}. The scales $g_N$ and $R_N$ are (up to sub-exponential
  prefactors) the same as previously, including the condition
\eqref{cond:alphabeta} or the range of parameters $\alpha ,\beta $. As in
\cite{bAC08}, the right inequality in \eqref{cond:alphabeta} is completely
natural, beyond it $Y$ `feels' the finiteness of $\HH_N$ and aging is not
expected to occur. The left inequality in \eqref{cond:alphabeta} is
technical, it ensures that the relevant deep traps are well separated
(cf.~\lemref{lem:sparsetop}), introducing certain simplifications in the
proof.
We believe that this bound might be improved to
$\alpha^2 \beta^2 /\beta_c^2 > 0$, by further exploiting our
method.
Finally, as previously, note
that \eqref{cond:alphabeta} is satisfied also for $ \beta<\beta_c$ for
appropriate $\alpha$, hence aging can occur above the critical
temperature.

\smallskip

2. Our choice of the fast chain $Y$ is rather unusual. In view of the previous
papers \cite{MM12,BC11}, it would be natural to take instead the `uniform
chain' $\tilde Y$ with transition rates $\tau_x\wedge\tau_y$, that is without the correction
$1\wedge\tau_x$ which appears in \eqref{def:Y}. This chain is reversible with respect to the uniform
distribution on $\mathbb H_N$.
This choice has, however, some deficiencies. On the heuristic level,
$\tilde Y$ is not an acceleration of $X$, since it is much slower than $X$ on
sites with very small Gibbs measure $\tau_x\ll 1$. These sites, which are
irrelevant for the statics, then `act as
traps' on $\tilde Y$, making them relevant for the dynamics, which is
undesirable. On the technical level, the trapping on sites with small
Gibbs measure has the consequence that the
mixing time of $\tilde Y$ is very large.

Our choice of the fast chain $Y$ runs as fast as $X$ on
the sites with small Gibbs measure and thus does not have this deficiency.
Moreover, since
$\nu_x=Z_N^{-1}$ whenever $E_x\geq0$, the equilibrium distribution of the fast chain $Y$ is still uniform on
the relevant deep traps, so the clock process $S$ retains its usual
importance for aging.

Remark also that in order to overcome the similar difficulties, \cite{MM12}
truncate the Hamiltonian of the REM at $0$ which effectively sets
$\tau_x\ge 1$ for all $x\in \mathbb H_N$.
We prefer to retain the full REM and use the modified fast chain $Y$
instead. Finally, \cite{Gay14} uses the discrete skeleton of
$X$ as the base chain, which has some interesting features but introduces
similar undesirable effects.

\smallskip

3. We view \thmref{thm:main} as an aging statement, without further
considering any two-point correlation functions. Actually, it seems hard
to derive aging statements for the usual correlation functions from our
result without extending the paper considerably. Such derivation usually
requires some knowledge of the fast chain $Y$ that goes over the $M_1$-convergence
of the clock processes. This knowledge is typically automatically
obtained in the previous approaches. The strength (or the weakness) of
our method is that we do not need to obtain such finer knowledge to show
the clock process convergence.

\smallskip

4. A rather unusual feature of Theorem~\ref{thm:main} is the fact that the
scaling $R_N$ is random, it depends on the random environment. This again a
consequence of our technique.
Claim \eqref{e:Rass} in \thmref{thm:main} however shows that at
least the exponential growth of $R_N$ is deterministic.
The random scale $R_N$ is explicitly defined in \eqref{def:rN}. We will
see that its definition depends on a somewhat free choice of an auxiliary
parameter, but nevertheless the final result does not depend on this
parameter. Not only this property makes us conjecture that $R_N$ should
actually satisfy a deterministic law of large numbers,
\begin{equation*}
  \lim_{N\to \infty} h(N) e^{-\alpha^2 \beta^2N /2 } R_N =1, \qquad
  \mathbb P\text{-a.s.},
\end{equation*}
for some function $h(N)$ growing at most sub-exponentially.

\smallskip

5. The mode of convergence in \thmref{thm:main} is not optimal, one would rather
like to obtain the convergence in $P^{\tau}_{\nu}$-distribution for $\PP$-almost
every environment, which is usually called `quenched' convergence.
Actually, Theorem~\ref{thm:main} can be strengthened
slightly to a statement which is somewhere between $\mathbb P$-a.s.~convergence
and convergence in $\mathbb P$-probability. Namely,
the statement holds for a.e.~realization
of sites with `small' $\tau_x$, but only in probability over sites with
'large' $\tau_x$, cf.~Remark~\ref{rem:convmode}.

\smallskip

6. Our proof of Theorem~\ref{thm:main} strongly exploits the
i.i.d.~structure of the Hamiltonian of the REM. At present we do not know
if it is possible to combine our techniques with those used for the RHT
dynamics of the $p$-spin model in \cite{BBC09,BG12}.


\medskip

We proceed by commenting on the proof of \thmref{thm:main}, concentrating
mainly on its novelties. The general strategy so far to prove such a
result has been to first reduce the problem to the clock process
restricted to a set of deep traps which govern the behavior of the
original clock
process. The different methods then all more or less aim at dividing the
contribution of consequently found deep traps into essentially
i.i.d.~blocks. For example in \cite{bAC08} or \cite{BC11}, this is achieved by
controlling the hitting probabilities of deep traps, proving that they are
hit essentially uniformly in exponentially distributed times, and
controlling the time the chain spends at the deep traps by a sharp
control of the Green function. Similar rather precise estimates on
hitting probabilities and/or Green function are necessary in other
approaches. Using this i.i.d.~structure, one can then show convergence of
the clock process by standard methods, e.g.~computing the Laplace transform.

The method used in this paper is slightly inspired by the general
approach taken in \cite{FM14} and \cite{CW14}. There, models of trapped
random walks on $\ZZ^d$ are considered where few information about the
discrete skeleton as well as the waiting times of a continuous-time
Markov chain  are available, and minimal necessary conditions for
convergence of the clock process are found. Taking up this idea, instead
of analyzing in detail the behavior of the fast chain $Y$, we extract the
minimal amount of information needed to show convergence of the clock
process. In particular, we do not need any exact control of hitting
probabilities and Green functions of deep traps, as most previous work did.

The first step in our proof is standard, namely that the main
contribution to the clock process comes from a small set of vertices with
large Gibbs measure $\tau_x$, the so-called deep traps, and that in fact
the clock process of the deep traps converges to a stable subordinator.
Denote the set of deep traps by $\cD_N$ (see \secref{sec:defs} for
details). We will show that the clock process $S$ can be well approximated
by the `clock process of the deep traps'
\begin{equation}
  \label{e:SD}
  S_{\cD}(t) = \int_0^{t} (1\vee\tau_{Y_s})\Ind{Y_s\in \cD_N}ds.
\end{equation}
Then it remains to show that in fact $g_N^{-1}S_{\cD}(tR_N)$ converges to a
stable subordinator.

To this end, we will in some sense invert the standard procedure described
above. Instead of approximating the clock process by an i.i.d.~block
structure and then use the Laplace transform to show convergence, we will
first compute a certain conditional Laplace transform using some special
properties of the Metropolis dynamics. Then we analyze what is actually
needed in order to show convergence of the unconditional Laplace transform.

A bit more detailed, this will be done as follows. Under condition
\eqref{cond:alphabeta}, the deep traps are almost surely well separated.
This fact and the fact that the definition \eqref{def:Y} contains the factor
$\tau_x\wedge \tau_y$ imply that the transition rates $q_{xy}$ of the
fast chain $Y$ do not depend on the energies $E_x$ of the deep
traps, but only on their location. Therefore, one can condition on the
location of all traps and the energies $E_x$ of the non-deep traps,
which determines the law $P^{\tau}_{\nu}$ of $Y$, and take the expectation over the energies of
the deep traps. We call this a `quasi-annealed' expectation, and denote
it by $\EE_{\cD}$ for the moment. Let $\ell_  {t}(x)$ denote the local
time of the fast chain $Y$ (see \secref{sec:defs} for details). As
$\EE_{\cD}$ is simply an expectation over i.i.d.~random variables, the
quasi-annealed Laplace transform of the rescaled clock process of the
deep traps given $Y$ can be computed. It essentially behaves like
\begin{equation} \label{eq:quasiannealed}
  \EE_{\cD}\big[e^{-\lambda \frac{1}{g_N}S_{\cD}(tR_n)}\mid Y \big]
  \approx \exp\bigg\{-\cK \lambda^{\alpha} \varepsilon_N \sum_{x\in \cD_N}
    \ell_{tR_N}(x)^{\alpha}\bigg\}.
\end{equation}
Here, $\varepsilon_N$ is a deterministic sequence tending to 0 as $N\to\infty$.
The above approximation shows that the only object related to $Y$ we have to control is the
local-time functional $\varepsilon_N \sum_{x\in \cD_N}\ell_{tR_N}(x)^{\alpha}$.

We will show that this a priori non-additive functional of $Y$ actually behaves in
an additive way, namely that it converges to $t$ as $N\to\infty$, under
$P^{\tau}_{\nu}$ for $\PP$-a.e.~environment $\tau$. For this convergence to
hold it is sufficient to have some weak bounds on the mean hitting time of deep traps
as well as some control on the mixing of the chain $Y$ together with an
appropriate choice of the scale $R_N$ that depends on the environment.

Using standard methods we then strengthen the quasi-annealed convergence
to quenched convergence (in the sense of
\thmref{thm:main}).

\medskip

To conclude the introduction, let us comment on how our method might be extended.
The key argument in the computation of the quasi-annealed Laplace
transform, namely the fact that the chain $Y$ is independent of the depth of
the deep traps, seems very specific for the Metropolis dynamics. However,
by adapting the method appropriately and using network reduction techniques,
we believe that one could also treat Bouchaud's asymmetric dynamics and
Metropolis dynamics in the regime where the left-hand side inequality of
\eqref{cond:alphabeta} fails, i.e.~there are neighboring deep traps.

\medskip

The rest of the paper is structured as follows. Detailed definitions
and notations used through the paper are introduced in \secref{sec:defs}. In
\secref{sec:mix} we analyze the mixing properties of the fast chain $Y$,
which will be crucial at several points later.  In
\secref{sec:meanhit} we give bounds on the mean hitting time of deep traps
and on the normalizing scale $R_N$. Using these bounds and the results on the
mixing of $ Y$, we show concentration of the local time functional
$\varepsilon_N\sum_{x \in\cD_N}\ell_{tR_N}(x)^{\alpha}$ in
\secref{sec:localtimes}. We prove convergence of the rescaled clock process
of the deep traps in \secref{sec:deep} with the above mentioned computation
of the quasi-annealed Laplace transform, using the concentration of the local
time functional. Finally, we treat the shallow traps in \secref{sec:shallow}
by showing that their contribution to the clock process can be neglected. In
Appendix~\ref{appendix}  we give the proof of a technical result which is used to
bound the expected hitting times in \secref{sec:meanhit}.


\section{Definitions and notation} 
\label{sec:defs}

In this section we introduce some notation used through the paper and
recall a few useful facts. We use
$\mathbb H_N$ to denote the $N$-dimensional hypercube $\{-1,1\}^N$
equipped with the usual distance
\begin{equation*}
  d(x,y)=\frac{1}{2}\sum_{i=1}^N |x_i-y_i|,
\end{equation*}
and write $\mathcal E_N$ for the set of nearest-neighbor edges
$\cE_N = \{\{x,y\}:~d(x,y)=1\}$.

For given parameters $\alpha$ and $\beta$,  let
\begin{equation}
  \label{e:gamma}
  \gamma=\frac{\alpha^2 \beta^2}{\beta_c^2}\in ( 1/2, 1),
\end{equation}
by condition \eqref{cond:alphabeta} in \thmref{thm:main}.

Recall from the introduction that $(E_x:x\in\HH_N,N\ge 1)$, is a
family of i.i.d.~standard Gaussian random variables defined on some
probability space $(\Omega,\cF,\PP)$. Note that we do not denote the
dependence on $N$ explicitly, but we assume that the space
$(\Omega , \mathcal F, \mathbb P)$ is the same for all $N$. For $\beta>0$ the
non-normalized Gibbs factor $\tau_x$ is given by
$\tau_x=e^{\beta\sqrt{N}E_x}$.

Using the standard  Gaussian tail approximation,
\begin{equation}
  \label{eq:gaussapprox}
  \PP[E_x\geq t] = \frac{1}{t\sqrt{2\pi} }\ e^{-{t^2}/{2}}\big(1+o(1)\big)
  \quad \text{as }t\to\infty,
\end{equation}
we obtain that $g_N$, as defined in \thmref{thm:main}, satisfies
\begin{equation*}
  \label{e:tautail}
  \PP[\tau_x>ug_N] = u^{-\alpha} 2^{-\gamma N}\big(1+o(1)\big).
\end{equation*}
This heuristically important computation explains the appearance of stable
laws in the distribution of sums of $\tau_x$: If we
observe $2^{\gamma N}$ vertices, then finitely many of them have their
rescaled Gibbs measures $\tau_x/g_N$ of order unity, and, moreover,
those rescaled Gibbs measures behave like random variables in the domain
of attraction of an $\alpha$-stable law.

Recall also that $Y=(Y_t)_{t\ge 0}$ stands for the fast Markov chain whose
transition rates $q_{xy}$ are given in \eqref{def:Y}, and that
$\nu = (\nu_x)_{x\in \mathbb H_N}$ denotes the invariant distribution of
this chain,  $\nu_x=\frac{1\wedge\tau_x}{Z_N}$.
For a given environment $\tau=(\tau_x)_{x\in\HH_N}$, let
$P^{\tau}_x$ and $P^{\tau}_{\nu}$ denote the laws of $Y$ started from a vertex
$x$ or from $\nu $ respectively, and $E^{\tau}_x$, $E^{\tau}_{\nu}$
the corresponding expectations.

Note that the normalization factor
$Z_N = \sum_{x\in \mathbb H_N} (1\wedge\tau_x)$ satisfies, for every
constant $\kappa \in (0,1/2)$,
\begin{equation}
  \label{e:boundzN}
  \kappa2^N \leq Z_N \leq 2^N \qquad \mathbb P\text{-a.s for $N$
    large enough}.
\end{equation}
Indeed, obviously $Z_N \leq 2^N$, and
$Z_N \geq \sum_{x\in\HH_N} \Ind{E_x\geq0}$. But $\Ind{E_x\geq0}$ are
i.i.d.~Bernoulli random variables, therefore the statement follows
immediately by the law of large numbers.

An important role in the study of properties
of $Y$ is played by the conductances defined by
\begin{equation} \label{def:cond}
  c_{xy} = \nu_x q_{xy} =  \frac{\tau_x\wedge \tau_y}{Z_N} \qquad\text{for }x\sim y.
\end{equation}

Let $\theta_s$ be the left shift on the space of trajectories of $Y$, that
is
\begin{equation}
  \label{e:theta}
  (\theta_sY)_t=Y_{s+t}.
\end{equation}
Let $H_x = \inf\{t>0:~Y_t=x\}$ be the hitting time of $x$ by $Y$,  $J_1$
the time of the first jump of $Y$, and let
$H^+_x = H_x\circ \theta_{J_1} + J_1 = \inf\{t>J_1:~Y_t=x\}$
be the return time to $x$ by $Y$.
Similarly define $H_A$ and $H^+_A$ for a set
$A \subset\HH_N$. The local time $\ell_t(x)$ of $Y$ is given by
\begin{equation*}
  \ell_t(x) = \int_0^t\Ind{Y_s=x}ds.
\end{equation*}
Using this notation the clock process $S$ introduced in \eqref{e:clock} can
be written as
\begin{equation*}
  S(t) = \int_0^t (1\vee\tau_{Y_s})ds = \sum_{x\in\HH_N}
  \ell_t(x) (1\vee\tau_x).
\end{equation*}

To define the set of deep traps $\cD_N$ and the random scale $R_N$
mentioned in the introduction we introduce a few additional parameters.
For $\alpha\in(0,1)$, $\beta>0$  as in
\thmref{thm:main} and
$\gamma$ as in \eqref{e:gamma}, we fix $\gamma '$ and $\alpha '$ such
that
\begin{equation}
  \label{e:gammarange}
  \frac{1}{2}<\gamma'<\gamma, \quad \text{and} \quad
  \alpha' = \frac{\beta_c}{\beta} \sqrt{\gamma'}.
\end{equation}
An explicit choice of $\gamma'$ will be made later in
\secref{sec:localtimes}.
We define the auxiliary scale
\begin{equation*}
  g'_N = e^{\alpha' \beta^2 N}
  (\alpha'\beta\sqrt{2\pi N})^{-\frac{1}{\alpha' }},
\end{equation*}
and set
\begin{equation*}
  \cD_N=\{x\in\HH_N:~\tau_x \geq g'_N\}.
\end{equation*}
to be the set of deep traps. By the Gaussian tail approximation
\eqref{eq:gaussapprox} it follows that the density of $\cD_N$ satisfies
\begin{equation}
  \label{eq:density}
  \PP[x\in \cD_N] = 2^{-\gamma' N}(1+o(1)).
\end{equation}

We quote the following observation on the size and sparseness of $\cD_N$.
The sparseness will play a key role in our computation of the
quasi-annealed Laplace transform in \secref{sec:deep}.
\begin{lemma}{\cite[Lemma~3.7]{bAC08}}
  \label{lem:sparsetop}
  For every $\varepsilon>0$,
  $\PP$-a.s.~for $N$ large enough,
  \begin{equation} \label{eq:sizetop}
    |\cD_N|2^{(\gamma'-1)N} \in (1-\varepsilon,1+\varepsilon).
  \end{equation}
  Moreover, since $\gamma'>1/2$, there exists $\delta>0$ such that
  $\PP$-a.s.~for $N$ large enough, the separation event
  \begin{equation} \label{eq:disttop}
    \mathscr{S}=\left\{\min\{d(x,y):~x,y\in \cD_N\} \geq \delta N\right\}
  \end{equation}
  holds.
\end{lemma}

Finally, for the sake of concreteness, let us give the explicit form of the
random scale $R_N$,
\begin{equation} \label{def:rN}
  R_N = 2^{(\gamma-\gamma')N}\left(\sum_{x\in \cD_N}
    \frac{E^{\tau}_x[\ell_{\tx}(x)^{\alpha}]}{E^{\tau}_{\nu}[H_x]}\right)^{-1},
\end{equation}
where $\tx$ denotes the mixing time of $Y$, a randomized stopping time
which we will construct in \secref{sec:mix}.
The reason for this definition will become apparent when we prove the
concentration of the local time functional mentioned in the introduction.
Although the definition of $R_N$ seems arbitrary by the somewhat free choice
of the parameter $\gamma'$, \thmref{thm:main} actually shows that
asymptotically $R_N$ will be independent of $\gamma'$.

\medskip

For the rest of the paper, $c,c',c''$ will always denote positive constants
whose values may change from line to line. We will use the notation $g=o(1)$
for a function $g(N)$ that tends to $0$ as $N\to\infty$, and $g=O(f)$ for a
function $g(N)$ that is asymptotically at most of order $f(N)$,
i.e.~$\lim_{N\to\infty}|g(N)|/f(N)\leq c$, for some $c>0$.


\section{Mixing properties of the fast chain} 
\label{sec:mix}

The fact that the chain $Y$ mixes fast, namely on a scale polynomial in $N$,
plays a crucial role in many of our arguments. In this section we
analyze the mixing behavior of $Y$. We first give a lower bound on the
spectral gap $\lambda_Y$ of $Y$, which we then use to construct a strong
stationary time $\tx$.

\begin{proposition}
  \label{prop:spectralgap}
  There are constants $\kappa>0$, $K>0$, $C_0>0$, such that
	$\PP$-a.s.~for $N$ large enough,
  \begin{equation*}
    \lambda_Y \geq \frac{\kappa}{4} N^{-K-1-\beta C_0}.
  \end{equation*}
\end{proposition}

We prove this proposition with help of the Poincaré inequality derived
in \cite{DS91}. To state this inequality, let $\Gamma$ be a
complete set of self-avoiding nearest-neighbor paths on $\HH_N$, that is for
each $x\neq y\in\HH_N$ there is exactly one path
$\gamma_{xy} \in \Gamma$ connecting $x$ and $y$. Let $|\gamma|$ be the
length of the path $\gamma$.
By Proposition 1' of \cite{DS91}, using also the reversibility of
$Y$ and recalling the definition \eqref{def:cond} of the conductances, it
follows that
\begin{equation}
  \label{eq:ds91}
  \frac{1}{\lambda_Y} \leq
  \max_{e=\{u,v\}\in\cE_N} \Bigg\{\frac{1}{c_{uv}}\sum_{
		\substack{\gamma_{xy}\in\Gamma:\\ \gamma_{xy}\ni e}}
    |\gamma_{xy}|\nu_x\nu_y \Bigg\}.
\end{equation}

To minimize the right-hand side of \eqref{eq:ds91}, a special care should be
taken of the edges whose conductance $c_{uv} = (\tau_u \wedge \tau_v)/Z_N$ is
very small, that is which are incident to vertices with very small $\tau_u$.
Those `bad' edges should be avoided if possible by paths $\gamma \in \Gamma$.
They cannot be avoided completely, since $\Gamma$ should be a complete set
of paths. On the other hand, if such edge is the first or the last edge of
some path $\gamma_{xy} $, its small conductance is canceled by equally small
$\nu_x$ or $\nu_y$. Therefore, to apply \eqref{eq:ds91} efficiently, one
should find a set of paths $\Gamma$ such that all paths $\gamma\in\Gamma$ avoid
`bad' vertices, except for vertices at both ends of the paths.

In the context of spin glass dynamics this method was used before in
\cite{FIKP98} to find the spectral gap of the Metropolis dynamics
\eqref{def:X}. Using the same approach, that is using the same set of paths
$\Gamma$ as in \cite{FIKP98},
we could find a lower bound on the spectral gap of the fast chain $Y$ of
leading order $\exp\{-c\sqrt{ N \log N}\}$. This turns out to be too
small for our purposes, cf.~\remref{rem:convmode}.

In the next lemma we construct a set of paths $\Gamma$ that avoids more `bad'
vertices, which allows to improve the lower bound on the spectral gap to
be polynomial in $N$. This is possible by using an embedding of $\HH_N$ into
its sub-graph of `good' vertices, i.e.~vertices with
not too small $\tau_x$, which is inspired by similar embeddings in
\cite{HLN87}.

For a nearest-neighbor path $\gamma = \{x_0,\dots,x_n\}$,
we call the vertices $x_1,\dots,x_{n-1}$ the interior vertices of
$\gamma$, and the edges $\{x_i,x_{i+1}\}$, $i=1,\dots,n-2$, the interior
edges of $\gamma $.

\begin{lemma}
  \label{l:goodpaths}
  There is an integer $K>0$ and a constant $C_0>0$, such that
	$\PP$-a.s.~for $N$ large
  enough there exists a complete set of paths $\Gamma$, such that
  the following three properties hold.
  \begin{enumerate}[(i)]
    \item	For every path $\gamma\in\Gamma$, every interior edge
    $e=\{u,v\}$ satisfies
    \begin{equation*}
      Z_N c_{uv}=\tau_u\wedge\tau_v \geq N^{-\beta C_0}.
    \end{equation*}
    \item $|\gamma|\leq 8 N$ for all $\gamma\in\Gamma$.
    \item Every edge $e\in\cE_N$ is contained in at most
    $N^K 2^{N-1}$ paths $\gamma\in\Gamma$.
  \end{enumerate}
\end{lemma}

\begin{proof}
  For $C_0>0$, whose value will be fixed later, we say that $x\in\HH_N$
  is good if $\tau_x \geq N^{-\beta C_0}$, and it is bad otherwise. To construct
  the complete set of paths $\Gamma$ satisfying the required properties,
  we will use the fact that the set of good vertices is very dense in $\HH_N$.
  In particular, we will show that
  \begin{equation} \label{eq:goodneigh}
    \parbox{0.85\textwidth}{$\PP$-a.s.~for $N$ large enough, every $x\in\HH_N$
      has at least
      $\frac{1}{2}C_0\sqrt{N}$ good neighbors,}
  \end{equation}
  and
  \begin{equation}\label{eq:goodpath}
    \parbox{0.85\textwidth}{$\PP$-a.s.~for $N$ large enough, for any pair of
      vertices $x,y$ at
      distance 2 or 3, there is a nearest-neighbor path of length at most
      7 connecting $x$ and $y$, such that all interior
      vertices of this path are good,}
  \end{equation}

  To prove these two claims, note first that for any $x\in\HH_N$, the
  probability of being bad is
  \begin{equation*}
    \PP\big[\tau_x < N^{-\beta C_0}\big]
    = \PP[E_x < - C_0 N^{-\frac{1}{2}}\log N]
    = \frac{1}{2} - \int_0^{C_0 N^{-\frac{1}{2}}\log N}
    \frac{1}{\sqrt{2\pi}} e^{-\frac{s^2}{2}}ds.
  \end{equation*}
  For $N$ large enough the integrand is larger than $\frac{1}{2}$, and it
  follows that
  \begin{equation*}
    \PP[x\text{ is bad}]
    \leq \frac{1}{2}\big(1-C_0N^{-\frac{1}{2}}\log N\big)
    =: \frac 12(1-q_N).
  \end{equation*}
  Hence, the number of
  bad neighbors of a vertex $x\in\HH_N$ is stochastically dominated by a
  Binomial$\big(N,\frac{1}{2}(1-q_N)\big)$ random variable $B$. For
  $\lambda>0$, the exponential Chebyshev inequality yields
  \begin{align*}
    \PP\big[x &\text{ has more than } N-\frac{1}{2} C_0\sqrt{N}
      \text{ bad neighbors}\big]\\
    &\leq \PP\big[B\geq N-\frac{1}{2} C_0\sqrt{N}\big]
    = \PP\big[e^{\lambda B} \geq e^{\lambda (N-\frac{1}{2} C_0\sqrt{N})}
      \big] \\
    &\leq e^{-\lambda (N-\frac{1}{2} C_0\sqrt{N})} \Big(1+
      \frac{1}{2}(1-q_N)(e^{\lambda}-1)\Big)^N \\
    &= e^{-\lambda (N-\frac{1}{2} C_0\sqrt{N})}
    \bigg( \frac{e^{\lambda}}{2}\Big(1-q_N	+ e^{-\lambda}(1+q_N)\Big)
      \bigg)^N \\
    &\leq 2^{-N} e^{\frac{\lambda}{2} C_0\sqrt{N}}
    \Big( \exp\{-q_N + e^{-\lambda}(1+q_N)\}\Big)^N.
  \end{align*}
  Since $q_N\to 0$ as $N\to\infty$, the last term in the parenthesis is
  bounded by $2e^{-\lambda}$ for $N$ large enough. Inserting $q_N$ and
  choosing $\lambda = \log N $, the above is bounded by
  \begin{align*}
    2^{-N} & \exp\Big\{\frac{1}{2} C_0\sqrt{N}\log N
      - C_0 \sqrt{N} \log N + 2\Big\}\\
    &\leq 2^{-N} \exp\Big\{-\frac{1}{4}C_0\sqrt{N} \log N\Big\},
  \end{align*}
  for $N$ large enough. With a union bound over all $x\in\HH_N$
  and using the Borel-Cantelli lemma, \eqref{eq:goodneigh} follows.

  To prove \eqref{eq:goodpath}, we first introduce some notation.
  For a given vertex $x$ and $\{i_1,\dots,i_k\}\subset \{1,\dots,N\}$, denote by
  $x^{i_1\cdots i_k}$ the vertex that differs from $x$ exactly in coordinates
  $i_1,\dots,i_k$. If two vertices $x$ and $y$ are at distance $2$, then
  $y=x^{kl}$ for some $k,l\in \{1,\dots, N\}$. Then for
  $\{i,j\}\cap \{k,l\}=\emptyset$ we define the path
  $\gamma_{xy}^{ij}$ of length 6 as
  $\{x,x^i, x^{ij}, x^{ijk}, x^{ijkl}=y^{ij}, y^{j}, y\}$.
  Similarly, for $x,y$ with $d(x,y)=3$, we have $y=x^{klm}$, and for
  $\{i,j\}\cap \{k,l,m\}=\emptyset$ we define
  the path $\gamma_{xy}^{ij}$ of length 7 by
  $\{x,x^i,x^{ij},x^{ijk},x^{ijkl},x^{ijklm}=y^{ij},y^j,y\}$.
  Observe that for fixed $x,y$ with $d(x,y)=2$ or $3$ and for different pairs
  $i,j$ the innermost 3 or 4 vertices of the paths $\gamma_{xy}^{ij}$ are disjoint.

  We now show that with high probability, for every $x,y$ at distance $2$ or
  $3$, we may find $i,j$ such that $\gamma_{xy}^{ij}$ has only good interior
  vertices.
  Fix a pair $x,y\in\HH_N$ at distance 2 or 3, and let as above $k,l$ or
  $k,l,m$ be the
  coordinates in which $x$ and $y$ differ. Assume for the moment that both $x$
  and $y$ have at least $\frac{1}{2}C_0\sqrt{N}$ good neighbors. Then there
  are at least $\frac{1}{4}C_0^2 N$ pairs $i,j$ such that the vertices
  $x^i$ and $y^j$ are good. Moreover, since it is a matter of dealing with a
  constant number of exceptions, we may tacitly assume that $i\neq j$, and
  $\{i,j\}\cap\{k,l\}=\emptyset$ or $\{i,j\}\cap\{k,l,m\}=\emptyset$,
  respectively.

  The remaining interior vertices $\{x^{ij}, x^{ijk}, x^{ijkl}=y^{ij}\}$ or
  $\{x^{ij},x^{ijk},x^{ijkl},x^{ijklm}=y^{ij}\}$ are all good with probability
  strictly larger than $1/2$, so the probability that one or more of
  these vertices are bad is bounded by $15/16$. Since these 3 or 4
  innermost vertices are disjoint for different pairs $i,j$, by
  independence, the
  probability that among all $\frac{1}{4}C_0^2 N$ pairs $\{i,j\}$ there is none for which
  all innermost 3 or 4 vertices of $\gamma_{xy}^{ij}$ are good is bounded
  by $(15/16)^{\frac{1}{4}C_0^2 N}$.
  Hence, for one fixed pair $x,y\in\HH_N$ at distance 2 or 3, where
  both $x$ and $y$ have at least $\frac{1}{2}C_0\sqrt{N}$ good neighbors,
  the probability that there is no path from $x$ to $y$
  of length 6 or 7 with all interior vertices good is bounded by
  \begin{equation*}
    (15/16)^{\frac{1}{4}C_0^2 N}.
  \end{equation*}
  There are less than $2^N (N^2+N^3)$ pairs of vertices at distance 2 or 3
  respectively, and we know from the proof of \eqref{eq:goodneigh} that with
  probability larger than $1- e^{-c \sqrt{N}\log N}$ every $x\in\HH_N$ has at
  least $\frac{1}{2}C_0\sqrt{N}$ good neighbors. It follows that the
  probability that the event in \eqref{eq:goodpath} does not happen is bounded by
  \begin{equation}\label{eq:union2}
    e^{-c \sqrt{N}\log N}
    + 2^N (N^2+N^3) (15/16)^{\frac{1}{4}C_0^2 N}.
  \end{equation}
  Choosing $C_0>\sqrt{\frac{4\log2}{\log15/16}}$ and applying the Borel-Cantelli
  lemma implies \eqref{eq:goodpath}.

  We now use the density properties \eqref{eq:goodneigh} and
  \eqref{eq:goodpath} of good vertices to define a (random)
  mapping from the hypercube to its sub-graph of good vertices. Let
  \begin{equation*}
    \cP_N=\big\{\{x_0,\dots,x_k\}:~k\geq0,~d(x_i,x_{i-1})
      =1~\forall~i=1,\dots,k\big\}
  \end{equation*}
  be the set of finite nearest-neighbor paths on $\HH_N$, including paths of
  length zero, which are just single vertices. Define the mapping
  $\varphi_N:\{\HH_N,\cE_N\}\to\{\HH_N,\cP_N\}$ in the following way.
  For $x\in\HH_N$, let
  \begin{equation*}
    \varphi_N(x) = \begin{cases} x, &\text{if $x$ is good;} \\
      \text{$x^i$},
      &\text{if $x$ and $x^j$, $j<i$, are bad but $x^i$ is good}; \\
      x, &\text{if $x$ is bad and has no good neighbor}.
    \end{cases}
  \end{equation*}
  By \eqref{eq:goodneigh}, $\PP$-a.s.~for $N$ large enough the last option will
  not be used, and therefore  $\varphi_N$ maps all vertices to good vertices. In this case,
  for two neighboring vertices $x,y$, their good images $\varphi_N(x)$ and
  $\varphi_N(y)$ can either coincide, or be at distance 1, 2, or 3.

  For an edge $e=\{x,y\} \in \cE_N$, let $\varphi_N(e)$ be
  \begin{itemize}
    \item the `path' $\{\varphi_N(x)\}$, if $\varphi_N(x)$ is good and
    $\varphi_N(x)=\varphi_N(y)$;

    \item the path $\{\varphi_N(x),\varphi_N(y)\}$, if both $\varphi_N(x)$ and
    $\varphi_N(y)$ are good and at distance 1;

    \item the path $\gamma^{ij}_{\varphi_N(x),\varphi_N(y)}$ with `minimal'
    $i,j$ such that all vertices of this path are good, if both
    $\varphi_N(x)$ and $\varphi_N(y)$ are good with distance 2 or 3 and
    such path exists;

    \item the path $\{x,y\}$ in any other case.
  \end{itemize}
  From \eqref{eq:goodneigh} and \eqref{eq:goodpath} it follows that
  $\PP$-a.s.~for $N$ large enough the last option does not occur and
  $\varphi_N$ maps all edges to paths that contain only good vertices.

  Finally, we extend $\varphi_N$ to be a map that sends paths to paths. For
  $\gamma =\{x_0,\dots,x_n\}\in \mathcal P_N$ we define $\phi_N(\gamma )$ to be a
  concatenation of paths $\phi_N(\{x_{i-1},x_i\})$, $i=1,\dots, n$, with
  possible loops erased by an arbitrary fixed loop-erasure algorithm.
  Note that $\varphi_N$ can make paths shorter or longer, but by construction, for
  any path $\gamma\in\cP_N$,
  \begin{equation} \label{eq:pathlength}
    |\varphi_N(\gamma)|\leq 7|\gamma|.
  \end{equation}

  We can now construct the random set of paths $\Gamma$ that satisfies the
  properties of the lemma. We first define a certain canonical set
  of paths $\tilde{\Gamma}$, and then use the mapping $\phi_N$ to construct
  $\Gamma$ from $\tilde{\Gamma}$.

  For any pair of vertices $x\neq y\in\HH_N$, let
  $\tilde{\gamma}_{xy}$ be the path from $x$ to $y$ obtained by consequently
  flipping the disagreeing coordinates, starting at coordinate 1. These paths
  are all of length smaller or equal to $N$, and the set
  $\tilde{\Gamma}=\{\tilde{\gamma}_{xy}:~x\neq y\in\HH_N\}$ has
  the property that any edge $e$ is used by at most $2^{N-1}$ paths in
  $\tilde{\Gamma}$. Indeed,
  if $e=\{u,v\}$, then there is a unique $i$ such that $u_i\neq v_i$. By
  construction, $e\in \tilde \gamma_{xy}$ if
  \begin{align*}
    x&=(x_1,\dots,x_{i-1}, u_i, u_{i+1},\dots,u_N),\\
    y&=(v_1,\dots,v_{i-1},v_i,y_{i+1},\dots,y_N).
  \end{align*}
  It follows that a total of $N-1$ coordinates of $x$ and $y$ are
  unknown, and so the number of possible pairs $x,y$ for paths
  $\tilde{\gamma}_{xy}$ through $e$ is bounded by $2^{N-1}$ (cf.~\cite[Example~2.2]{DS91}).

  For any pair $x\neq y\in\HH_N$, let the path $\gamma_{xy}$ in the set
  $\Gamma$ be
  defined by
  \begin{equation*}
    \gamma_{xy} = \begin{cases}
      \phi_N(\tilde \gamma_{xy}),&\text{if $x,y$ are good},\\
      \{x\}\circ \phi_N(\tilde \gamma_{xy}),&\text{if $x$ is bad and $y$ is good},\\
      \phi_N(\tilde \gamma_{xy})\circ\{y\},&\text{if $x$ is good and $y$ is
        bad},\\
      \{x\}\circ \phi_N(\tilde \gamma_{xy})\circ\{y\},&\text{if $x$ is good and $y$ is
        bad},\\
    \end{cases}
  \end{equation*}
  where `$\circ$' denotes the path concatenation.

  It remains to check that this set of paths $\Gamma$ indeed satisfies the
  required properties. First, by construction, $\Gamma$ is complete, that
  is every path $\gamma_{xy}\in\Gamma$ connects $x$ with $y$ and is
  nearest-neighbor and self-avoiding. Further, by construction of $\varphi_N$ and
  the properties \eqref{eq:goodneigh} and \eqref{eq:goodpath}, $\PP$-a.s.~for
  $N$ large enough, all interior vertices of all $\gamma\in\Gamma$
  are good, i.e.~(i) is satisfied. Moreover, by
  \eqref{eq:pathlength} and the construction of the paths
  $\tilde\gamma\in\tilde{\Gamma}$, the paths $\gamma\in\Gamma$ have
  length at most $7N+2$, hence (ii) is satisfied for $N\geq2$. Finally,
  $\varphi_N$ deforms the paths $\tilde{\gamma}\in\tilde{\Gamma}$ only
  locally, so that the number of paths in $\Gamma$ passing through an edge
  $e$ is bounded by the number of paths in $\tilde{\Gamma}$ passing through
  the ball of radius 4 around $e$. But this number is bounded by $2^{N-1}$
  times the number of edges in that ball, which is bounded by $N^K$ for
  some integer $K>0$. This proves (iii) and thus finishes the proof of the
  lemma.
\end{proof}

We can now prove the spectral gap estimate.

\begin{proof}[Proof of \propref{prop:spectralgap}]
  $\PP$-a.s.~for every $N$ large enough we can find a complete set of
  paths $\Gamma$ such that (i), (ii) and (iii) of \lemref{l:goodpaths}
  and \eqref{e:boundzN} hold. By (ii), the expression in \eqref{eq:ds91}
  over which the maximum is taken is bounded from above by
  \begin{equation} \label{eq:exprmax}
    \frac{8N}{Z_N} \frac{1}{\tau_u \wedge \tau_v} \sum_{\gamma_{xy} \ni \{u,v\}}
    (\tau_x\wedge 1)(\tau_y \wedge 1).
  \end{equation}
  We distinguish three cases for the position of the edge $\{u,v\}$ in a
  path $\gamma_{xy}$.

  \begin{enumerate}[(1)]
    \item If $\{u,v\}$ is an interior edge of
    $\gamma_{xy}$, then  $\tau_u \wedge \tau_v$ is larger than
    $N^{-\beta C_0}$ by (i) of \lemref{l:goodpaths}.

    \item If $\{u,v\}$ is at the end of the path $\gamma_{xy}$, say at $u=x$,
    and $v$ is an interior vertex of $\gamma_{xy}$, then
    $\tau_x \wedge \tau_v$ is either larger than
    $N^{-\beta C_0}$, or it is equal to $\tau_x$ in which
    case it cancels with $\tau_x\wedge1$. Indeed, if $\tau_x \wedge \tau_v$
    was smaller than $N^{-\beta C_0}$ and equal to
    $\tau_v$, then $v$ would be a bad interior vertex of $\gamma$, which
    contradicts (i) of \lemref{l:goodpaths}.

    \item If $\gamma_{xy}$ only consists of the single edge $\{x,y\}$, then
    $\tau_x \wedge \tau_y$ is either larger than 1, or the term
    $\tau_x\wedge\tau_y$ cancels with the smaller one of $\tau_x\wedge1$
    and $\tau_y\wedge1$.
  \end{enumerate}
  It follows that
  for every edge $\{u,v\}$ the expression \eqref{eq:exprmax} is bounded from above by
  \begin{equation*}
    \frac{8N}{Z_N} N^{\beta C_0}
    \#\{\text{paths through }e\}.
  \end{equation*}
  Since, by (iii) of \lemref{l:goodpaths}, the number of paths is bounded
  by $N^K2^{N-1}$, and, by \eqref{e:boundzN},
  ${Z_N} \geq \kappa 2^{N}$, this completes the proof.
\end{proof}

In a next step we construct the mixing time $\tx$ of the fast chain $Y$. To
this end, define the mixing scale
\begin{equation}
  \label{e:mN}
  m_N = \frac{8}{\kappa}N^{K+3+\beta C_0}.
\end{equation}
Then \propref{prop:spectralgap} reads $\lambda_N \geq 2N^2m_N^{-1}$.

We assume that our probability space $(\Omega,\cF,\PP)$ is rich enough so
that there exist infinitely many independent uniformly on $[0,1]$ distributed
random variables, independent of anything else. A randomized stopping time $T$
 is a positive random variable such that the event $\{T\leq t\}$ depends only
on $\{Y_s:s\leq t\}$, the environment, and on the values of these additional random variables.

\begin{proposition} \label{prop:mixtime}
$\PP$-a.s.~for $N$ large enough, there exists a randomized stopping
time $\tx$ with values in $\{m_N,2m_N,3m_N,\dots\}$ such that $\tx$ is a
strong stationary time for $Y$, that is	for any (possibly random)
$Y_0\in\HH_N$,
\begin{enumerate}[(i)]
	\item $P^{\tau}_{Y_0}[Y_{\tx}=y]= \nu_y$,
	\item for any $k\geq1$, $P^{\tau}_{Y_0}[\tx\geq km_N] = e^{-(k-1)}$,
	\item $\tx$ and $Y_{\tx}$ are independent.
\end{enumerate}
\end{proposition}

\begin{proof}
This construction follows closely \cite[Proposition~3.1]{MM12}, with only
minor adaptations.
Define the following distances from stationarity,
\begin{equation*}
  \begin{split}
    s(t) &= \min\{s\geq 0:~\forall x,y\in\HH_N,~P^{\tau}_x[Y_{t}=y]
      \geq (1-s)\nu(y)\},\\
    \bar{d}(t) &=\max_{x,y\in\HH_N} \|P^{\tau}_x[Y_{t}\in\cdot\,]
    - P^{\tau}_y[Y_{t}\in\cdot\,]\|_{TV},
  \end{split}
\end{equation*}
where $\|\cdot\|_{TV}$ denotes the total variation distance. Define the time
\begin{equation*}
	\cT=\inf\{t\geq0:~\bar{d}(t)\leq e^{-1}\}.
\end{equation*}
From \cite[Lemmas~4.5,~4.6~and~4.23]{AFb} we know that
\begin{equation}
  \begin{split}
    \label{e:mixingdistances}
    \bar{d}(t) &\leq e^{-\floor{t/\cT}},\\
    s(2t) &\leq 1-(1-\bar{d}(t))^2,\\
    \cT &\leq \frac{1}{\lambda_Y} \left(1+\frac{1}{2}\log \frac{1}{\nu^*}\right),
  \end{split}
\end{equation}
where $\nu^*=\min_x \nu_x$. Since $\PP[\tau_x\leq e^{-N^2}]\leq ce^{-c'N}$,
by the Borel-Cantelli lemma, $\PP$-a.s.~for $N$ large enough,
$\log \frac{1}{\nu^*} \leq N^2$. Therefore, by
Proposition~\ref{prop:spectralgap} and \eqref{e:mixingdistances},
$\PP$-a.s.~for $N$ large enough, $\cT \leq \frac{1}{2}m_N$,
$\bar{d}(\frac{1}{2}m_N)\leq e^{-1}$, and $s(m_N) \leq e^{-1}$, which means
that for all $Y_0,y\in\HH_N$,
\begin{equation*}
	P^{\tau}_{Y_0}[Y_{m_N} = y]\geq (1-e^{-1})\nu_y.
\end{equation*}

We can now define the strong stationary time $\tx$ with values in
$\{m_N,2m_N,\dots\}$. Let $U_1,U_2,\dots$ be i.i.d.~uniformly on $[0,1]$
distributed random variables, independent of anything else. Conditionally on
$Y_0=x$, $Y_{m_N}=y$,
let $\tx=m_N$ if
\begin{equation*}
	U_1\leq \frac{(1-e^{-1}) \nu_y}{P^{\tau}_x[Y_{m_N}=y]} \quad(\leq 1).
\end{equation*}
Otherwise, we define $\tx$ inductively: for every $k\in\NN$, conditionally on
$\tx>km_N$, $Y_{km_N}=z$ and $Y_{(k+1)m_N}=y$, let $\tx=(k+1)m_N$ if
\begin{equation*}
	U_{k+1}\leq \frac{(1-e^{-1}) \nu_y}{P^{\tau}_z[Y_{m_N}=y]} \quad(\leq 1).
\end{equation*}

By construction, we have for every $x\in \mathbb H_N$,
\begin{equation*}
	P^{\tau}_{x}[\tx=m_N\mid Y_{m_N}=y]
		= \frac{(1-e^{-1}) \nu_y}{P^{\tau}_{x}[Y_{m_N}=y]},
\end{equation*}
and thus
\begin{equation*}
	P^{\tau}_{Y_0}[\tx=m_N,~ Y_{m_N}=y\mid Y_0=x] = (1-e^{-1}) \nu_y.
\end{equation*}
Similarly, we have
\begin{equation*}
	P^{\tau}_{Y_0}[\tx=(k+1)m_N,~Y_{(k+1)m_N}=y\mid \tx>km_N,~Y_{km_N}=x]
		= (1-e^{-1}) \nu_y.
\end{equation*}
By induction over $k$, we obtain that for any $k\in\NN$ and $y\in\HH_N$,
\begin{equation*}
		P^{\tau}_{Y_0}[\tx=km_N,~ Y_{km_N}=y] = e^{-(k-1)}(1-e^{-1}) \nu_y,
\end{equation*}
which finishes the proof.
\end{proof}

For future reference we collect here two useful statements that follow directly
from the construction of $\tx$.
\begin{lemma}
  \label{l:tmixprop}
  For every $t>0$ and $x\in \mathbb H_N$ and every starting distribution
  $\rho $,
  \begin{equation*}
    \begin{split}
      P^\tau_\rho [Y_t=x| \tx<t]&=\nu_x,
      \\ \big|P^\tau_\rho [Y_t=x]-\nu_x\big| &\le
      P^\tau_\rho[\tx>t]=e^{-\floor{t/m_N-1}}.
    \end{split}
  \end{equation*}
\end{lemma}


\section{Bounds on mean hitting time and random scale}
\label{sec:meanhit}

In this section we prove bounds on the mean hitting time $E^{\tau}_{\nu}[H_x]$
of deep traps $x\in\cD_N$. As a corollary of the proof we will obtain a
useful bound on the Green function in deep traps. The bounds on the mean
hitting times will further imply bounds on the random scale $R_N$,
which will imply the claim \eqref{e:Rass} of \thmref{thm:main}.

\begin{proposition}\label{prop:meanhit}
There exists $\delta\in(0,1/6)$, such that $\PP$-a.s.~for $N$ large enough,
\begin{equation*}
	2^{N-N^{1-\delta}} \leq E^{\tau}_{\nu}[H_x] \leq 2^{N+N^{1-\delta}}
	\quad \text{for every $x\in\cD_N$}.
\end{equation*}
\end{proposition}

The proof of \propref{prop:meanhit} is split in two parts.

\begin{proof}[Proof of the upper bound]
For the upper bound we use \cite[Lemma~3.17]{AFb} which states that
\begin{equation*}
  E^{\tau}_{\nu}[H_x] \leq \frac{1-\nu_x}{\lambda_Y \nu_x}.
\end{equation*}
Since $\tau_x\geq1$ for deep traps $x\in\cD_N$, this is smaller than
$\frac{Z_N}{\lambda_Y}$, which by \propref{prop:spectralgap} and
\eqref{e:boundzN} is bounded by $2^{N+N^{1-\delta}}$, $\PP$-a.s.~for $N$
large enough.
\end{proof}

For the lower bound we will use a version of Proposition~3.2 of~\cite{CTW11}
which allows to bound the inverse of the mean hitting time
$E^{\tau}_{\nu}[H_x]$ in terms of the effective
conductance from $x$ to a suitable set $B$. Recall the definition of the
conductances $c_{xy}$ from \eqref{def:cond}, and let
$c_x=\sum_{y\sim x}c_{xy}$. Following the terminology of
\cite[Chapter~2]{LP}, we define the effective conductance between a vertex
$x$ and a set $B$ as
\begin{equation*}
	\cC(x\to B) = P^{\tau}_x[H^+_x > H_B] c_x.
\end{equation*}
By Proposition~\ref{prop:appendix}, which is a generalization of
\cite[Proposition~3.2]{CTW11} to arbitrary continuous-time
finite-state-space Markov chains,
\begin{equation}
  \label{eq:boundbycond}
  \frac{1}{E^{\tau}_{\nu}[H_x]} \leq \cC(x\to B) \nu(B)^{-2}.
\end{equation}

To apply this bound effectively, we should
find a set $B$ such that $\mathcal C(x\to B)$ is small and $\nu (B)$ close
to $1$.
In the next lemma we construct such sets $B$ for every
$x\in \mathbb H_N$. For these sets we
have some control on the conductances connecting $B$ and $B^c$. Using
standard network
reduction techniques  we can then give
a bound on the effective conductance $\cC(x\to B)$, which when plugged
into \eqref{eq:boundbycond} will imply the lower bound on
$E^{\tau}_{\nu}[H_x]$.

Denote by $B(x,r)=\{y\in\HH_N,~d(x,y)\leq r\}$ the ball of radius
$r$ around $x$, and by $\partial B(x,r)=\{y\in\HH_N,~d(x,y) = r\}$ the sphere
of radius $r$.

\begin{lemma}\label{lem:spheres}
For every $\delta\in (0,1/6)$, $\PP$-a.s.~for $N$ large enough, there
exist radii $(\rho_x)_{x\in\HH_N}$ satisfying $1\leq \rho_x\leq N^{3\delta}$,
such that for all $x\in\HH_N$ and for all $y\in\partial B(x,\rho_x)$,
$\tau_y \leq 2^{\frac{1}{2}N^{1-\delta}}$.
\end{lemma}

\begin{proof}
  Fix $\delta\in (0,1/6)$. We say that a sphere $\partial B(x,r)$ is good
  if  $\tau_y \leq 2^{\frac{1}{2}N^{1-\delta}}$ for all
  $y\in\partial B(x,r)$, otherwise we say that it is bad. Using the Gaussian
  tail approximation \eqref{eq:gaussapprox}, we get that
  \begin{equation*}
    \PP\big[\tau_y > 2^{\frac{1}{2}N^{1-\delta}}\big]
    \leq c e^{-\frac{\log^2 2}{8\beta^2}N^{1-2\delta}}.
  \end{equation*}
  The size of the sphere $\partial B(x,r)$ is bounded by $N^r$, hence the
  probability that the sphere $\partial B(x,r)$ is bad is bounded by
  \begin{equation*}
    N^r \PP\big[\tau_y > 2^{\frac{1}{2}N^{1-\delta}}\big]\leq
    c \exp\big\{r \log N - \frac{\log^2 2}{8\beta^2} N^{1-2\delta}\big\}.
  \end{equation*}
  By independence of the $\tau_x$, the probability that for one fixed $x$
  all the spheres $\partial B(x,r)$, $r=1,\dots,N^{3\delta}$, are bad is
bounded
  by
  \begin{align*}
    \prod_{r=1}^{N^{3\delta}} N^r \PP\big[\tau_y > 2^{\frac{1}{2}N^{1-\delta}}
      \big]
    &\leq \left(N^{N^{3\delta}} \PP\big[\tau_y > 2^{\frac{1}{2}N^{1-\delta}}
        \big]	\right)^{N^{3\delta}}\\
    &\leq	\exp\big\{N^{3\delta}\log c + N^{6\delta}\log N
      - \frac{\log^2 2}{8\beta} N^{1+\delta}\big\}.
  \end{align*}
  Finally, by a union bound, the probability that among all $2^N$ vertices in
  $\HH_N$ there is one for which all spheres
  $\partial B(x,r)$, $r=1,\dots,N^{3\delta}$, are bad is bounded by
  \begin{equation*}
    2^N \left(N^{N^{3\delta}}\PP\big[\tau_y>2^{\frac{1}{2}N^{1-\delta}}\big]
    \right)^{N^{3\delta}}
    \leq \exp\big\{N^{3\delta}\log c + N^{6\delta}\log N+N\log2
      -\frac{\log^2 2}{8\beta} N^{1+\delta}\big\}.
  \end{equation*}
  Since $\delta<1/6$ this decays faster than exponentially, and so by the
  Borel-Cantelli lemma the event occurs $\PP$-a.s.~only for finitely many $N$,
  i.e.~$\PP$-a.s.~for $N$ large enough we can find for every $x\in\HH_N$ a
	radius $\rho_x\leq N^{3\delta}$ such that the sphere $\partial B(x,\rho_x)$
	is good.
\end{proof}

\begin{proof}[Proof of the lower bound of \propref{prop:meanhit}]
  For every $x\in\cD_N$ we define the set $A_x=B(x,\rho_x)$ if the radius
  $\rho_x$ from \lemref{lem:spheres} exists, otherwise we take
  $A_x = \{x\}$. By \lemref{lem:spheres} and \eqref{e:boundzN}, $\PP$-a.s.~for $N$ large
  enough, for all $x\in\cD_N$ all conductances
  $c_{yz} = (\tau_y\wedge\tau_z)/Z_N$ connecting $A_x$ and $A_x^c$ are smaller
  than $2^{\frac{1}{2}N^{1-\delta}}/(\kappa 2^N)$.

  By the parallel law (cf.~\cite[Chapter~2.3]{LP}), the effective
  conductance between the boundaries of $A_x$ and $A_x^c$ is equal to the
  sum of all the conductances of edges connecting $A_x$ and $A_x^c$, and
  so $\PP$-a.s.~for $N$ large enough,
  \begin{equation*}
    \cC(\partial A_x\to\partial A_x^c) =
    \sum_{\substack{y\in\partial A_x \\ z \in\partial A_x^c}} c_{yz}
    \leq \kappa^{-1}N^{\rho_x+1} 2^{\frac{1}{2}N^{1-\delta}} 2^{-N}.
  \end{equation*}
  By Rayleigh's monotonicity principle (cf.~\cite[Chapter~2.4]{LP}), comparing
  the effective conductances from $x$ to $A_x^c$ before and after setting all
  the conductances inside $A_x$ to infinity, it follows that
  \begin{equation*}
    \cC(x\to A_x^c) \leq \cC(\partial A_x\to\partial A_x^c)
    \leq \kappa^{-1}N^{\rho_x+1} 2^{\frac{1}{2}N^{1-\delta}}2^{-N}.
  \end{equation*}
  Since $\delta<1/6$ and $\rho_x\leq N^{3\delta}$, we have
  $N^{\rho_x+1} \le 2^{\frac 12 N^{1-\delta}}$ for $N$ large enough, and
  thus, $\PP$-a.s.~for $N$ large enough,
  \begin{equation} \label{eq:effcond}
    \cC(x\to A_x^c) \leq c2^{-N+N^{1-\delta}}.
  \end{equation}
  Moreover, $\PP$-a.s.~for $N$ large enough, as
  $\nu_y = (1\wedge \tau_y)/Z_N \le 1/Z_N$, using \eqref{e:boundzN} again,
  \begin{equation}
    \label{e:Axsize}
    \nu(A_x^c) = 1-\nu (A_x)
    \geq 1-Z_N^{-1} |A_x| \ge 1- c 2^{-N} N^{N^{3\delta }}
    \xrightarrow{N\to \infty}1.
  \end{equation}
  Plugging \eqref{eq:effcond} and \eqref{e:Axsize} into
  \eqref{eq:boundbycond} and readjusting $\delta $ to accommodate for
  constants easily yields the required lower bound
  $E_\nu^\tau [H_x]\ge 2^{N-N^{1-\delta }}$. This completes the proof.
\end{proof}

As a corollary we get a lower bound on
$E^{\tau}_x[\ell_{H_{A_x^c}}(x)]$ for the deep traps $x\in \cD_N$.

\begin{corollary} \label{cor:green}
  There are constants $\delta\in(0,1/6)$ and $c>0$, such that $\PP$-a.s.~for $N$
  large enough, for all $x\in \cD_N$, under $P^{\tau}_x$ the local time of $Y$
  in $x$ before leaving $A_x$, $\ell_{H_{A_x^c}}(x)$, stochastically dominates
  an exponential random variable with mean $c2^{-N^{1-\delta}}$. In particular,
  $\PP$-a.s.~for $N$ large enough,
  \begin{equation*}
    E^{\tau}_x\big[\ell_{H_{A_x^c}}(x)\big] \geq c2^{-N^{1-\delta}}.
  \end{equation*}
\end{corollary}

\begin{proof}
  The local time at $x$ before hitting $A_x^c$ is an exponential random variable
  with mean equal to
  \begin{equation*}
    E^{\tau}_x\big[\#\{\text{visits to $x$ before $H_{A_x^c}$}\}\big]\cdot
    E^{\tau}_x[J_1].
  \end{equation*}
  The expected number of visits before leaving $A_x$ is
  $P^{\tau}_x[H^+_x > H_{A_x^c}]^{-1}=c_x\cC(x\to A_x^c)^{-1}$. The mean
  duration of
  one visit to $x$ is $E^{\tau}_x[J_1]=(\sum_{y\sim x}q_{xy})^{-1}$. For
  the deep traps we have $\tau_x>1$, therefore
  $\sum_{y\sim x}q_{xy}=\sum_{y\sim x} c_{xy}/\nu_x=Z_Nc_x$. It follows that the local
  time at $x$ before hitting $A_x^c$ is in fact an exponential random variable with
  mean $Z_N^{-1}\cC(x\to A_x^c)^{-1}$. Using the bounds
  \eqref{eq:effcond} and \eqref{e:boundzN}, the claim follows easily.
\end{proof}

As a next consequence we give bounds on the random scale $R_N$ defined in
\eqref{def:rN}. Note that this lemma also proves the statement
\eqref{e:Rass} about the asymptotic behavior of $R_N$ in \thmref{thm:main}.

\begin{lemma}\label{lem:boundrn}
  For every $\varepsilon>0$, $\PP$-a.s.~for $N$ large enough,
  \begin{equation*}
    2^{(\gamma-\varepsilon)N} \leq R_N \leq 2^{(\gamma+\varepsilon)N}.
  \end{equation*}
\end{lemma}

\begin{proof}
  By \propref{prop:mixtime}, $\tx/m_N$ is a geometric random variable with
  parameter $e^{-1}$, and thus
  $E^{\tau}_x[\ell_{\tx}(x)^{\alpha}] \leq E^{\tau}_x[\tx^{\alpha}]
  \leq c m_N^{\alpha}\leq e^{\epsilon N}$ by \eqref{e:mN}, for every
	$\epsilon>0$ and $N$ large enough.
	Moreover, $|\cD_N| \leq c' 2^{(1-\gamma')N}$ by
  \eqref{eq:sizetop}. Using the lower bound on $E^{\tau}_{\nu}[H_x]$ from
  \propref{prop:meanhit}, we obtain that for every $\varepsilon>0$,
	$\PP$-a.s.~for
  $N$ large enough,
  \begin{equation*}
    R_N = 2^{(\gamma-\gamma')N}\left(\sum_{x\in \cD_N}
      \frac{E^{\tau}_x[\ell_{\tx}(x)^{\alpha}]}
      {E^{\tau}_{\nu}[H_x]}\right)^{-1}
    \geq 2^{(\gamma-\varepsilon)N}.
  \end{equation*}

  For the upper bound we need a lower bound on
  $E^{\tau}_x[\ell_{\tx}(x)^{\alpha}]$. Recall the sets $A_x$ constructed
  in the proof of Proposition~\ref{prop:meanhit}, and note that
  \begin{equation}\label{eq:loctxalpha}
    E^{\tau}_x[\ell_{\tx}(x)^{\alpha}]
    \geq E^{\tau}_x\big[\Ind{\tx\geq H_{A_x^c}}\ell_{H_{A_x^c}}(x)^{\alpha}\big].
  \end{equation}
  By \corref{cor:green}, $\PP$-a.s.~for $N$ large enough, the local time at $x$
  before hitting $A_x^c$ stochastically dominates an exponential random
  variable with mean $c2^{-N^{1-\delta}}$, hence
  \begin{equation*}
    P^{\tau}_x\big[ \ell_{H_{A_x^c}}(x)
      \leq 2^{-2N^{1-\delta}}\big] \leq 1-e^{-c2^{-N^{1-\delta}}}
    \leq c2^{-N^{1-\delta}}.
  \end{equation*}
  Moreover, for every $\varepsilon>0$, $\PP$-a.s.~for $N$ large enough,
  \begin{equation*}
    P_x^{\tau}[\tx < H_{A_x^c}] \leq P_x^{\tau}[Y_{\tx}\in A_x]
    = \nu(A_x) \leq \kappa^{-1}2^{-N}N^{N^{3\delta}} \leq 2^{-\varepsilon N}.
  \end{equation*}
  Using the last two observations in \eqref{eq:loctxalpha}, $\PP$-a.s.~for
  $N$ large enough,
  \begin{align*}
    E^{\tau}_x&[\ell_{\tx}(x)^{\alpha}]
    \geq P^{\tau}_x\big[\{\tx\geq H_{A_x^c}\}\cap \{\ell_{H_{A_x^c}}(x)
        \geq 2^{-2N^{1-\delta}}\}\big] \left(2^{-2N^{1-\delta}}\right)^{\alpha } \\
    &\geq 2^{-2\alpha N^{1-\delta}}\left(P^{\tau}_x\big[\ell_{H_{A_x^c}}(x)
          \geq 2^{-2N^{1-\delta}}\big] - P^{\tau}_x\big[\{\ell_{H_{A_x^c}}(x)
          \geq 2^{-2N^{1-\delta}}\}\cap\{\tx < H_{A_x^c}\}\big]\right)\\
    &\geq 2^{-2\alpha N^{1-\delta}}\left(P^{\tau}_x\big[\ell_{H_{A_x^c}}(x)
        \geq 2^{-2N^{1-\delta}}\big] - P^{\tau}_x\big[\tx < H_{A_x^c}\big]
    \right)\\
    &\geq 2^{-2\alpha N^{1-\delta}}\left((1-c'2^{-N^{1-\delta}})
      - 2^{-\varepsilon N}\right)\\
    &\geq 2^{-\varepsilon N}.
  \end{align*}
  Combining this with $|\cD_N| \geq c 2^{(1-\gamma')N}$ by \eqref{eq:sizetop}
  and the upper bound on $E^{\tau}_{\nu}[H_x]$ from \propref{prop:meanhit}, we
  obtain the required upper bound on $R_N$.
\end{proof}


\section{Concentration of the local time functional}
\label{sec:localtimes}

In this section we prove the concentration of the local time functional
that appears in the computation of the quasi-annealed Laplace transform of
the clock process on the deep traps, as explained in the introduction
(cf.~\eqref{eq:quasiannealed}). We denote this functional by
\begin{equation*}
  L_N(t) = 2^{(\gamma'-\gamma) N} \sum_{x\in \cD_N}
  \ell_{tR_N}(x)^{\alpha}.
\end{equation*}

So far we had no restriction on the choice of $\gamma'$ other than
$1/2<\gamma '<\gamma$, see \eqref{e:gammarange}. We now make an explicit
choice as follows. Let
$\varepsilon_0=\frac{1}{2}\left((1-\gamma)\wedge(\gamma-\frac{1}{2})\right)$,
and define $\gamma'= \gamma-\varepsilon_0$, such that in particular
\begin{gather}
	1-\gamma \geq 2\varepsilon_0, \label{eq:epsilon0} \\
	\gamma-\gamma'=\varepsilon_0. \label{eq:gamma1}
\end{gather}

The main result of this section is the following proposition.

\begin{proposition}\label{prop:conclocal}
  For every fixed $t\ge 0$, $\PP$-a.s.~for $N$ large enough,
  \begin{equation*}
    P^{\tau}_{\nu}\left[\left| L_N(t) - t\right|
      \geq 2^{-\frac{1}{5}\varepsilon_0N}\right]
    \leq c2^{-\frac{1}{10}\varepsilon_0N}.
  \end{equation*}
\end{proposition}

\begin{proof}
  We approximate $L_N(t)$ by the sum of essentially independent
  random variables as follows. Let
  $K=\floor{2^{\varepsilon_0N}}$. For a fixed $t>0$, define
  \begin{equation*}
    t_k=\frac{tR_N}{K} k, \qquad k=0,\dots,K.
  \end{equation*}
  Recall the notation \eqref{e:theta}. For
  every $x\in\cD_N$ and
  $k=1,\dots,K$, define $H^k_x=t_{k-1}+H_x\circ \theta_{t_{k-1}}$ to be the
  time of the first visit to $x$ after $t_{k-1}$, and  set
  \begin{equation*}
    \ell_{t,x}^k = \left(\int_{H^k_x\wedge (t_{k}-2N^2m_N)}^{(H^k_x+N^2m_N)\wedge (t_{k}-N^2m_N)}
      \Ind{Y_s=x}ds\right)^{\alpha}.
  \end{equation*}
  The random variable $\ell_{t,x}^k$
  gives `roughly' the $\alpha$-th power of the time that $Y$ spends in $x$
  between $t_{k-1}$ and $t_k-N^2m_N$, with some suitable truncations.
  Let further
  \begin{equation*}
    U_N^k(t) = 2^{(\gamma'-\gamma)N}\sum_{x\in\cD_N} \ell_{t,x}^k.
  \end{equation*}

  The next lemma, which we prove later, shows that the sum of the
  $U_N^k(t)$'s is a good approximation for $L_N(t)$.
  \begin{lemma}
    \label{lem:ltxk}
    For every $t>0$, $\PP$-a.s.~for $N$ large enough,
    \begin{equation*}
      P_{\nu}^{\tau}\bigg[L_N(t)\neq\sum_{k=1}^K U_N^k(t)\bigg]
      \leq c2^{-\frac{1}{2}\varepsilon_0N}.
    \end{equation*}
  \end{lemma}

  With \lemref{lem:ltxk}, the proof of the proposition reduces to
  understanding of the approximating sum $\sum_{k=1}^K U_N^k(t)$. We will
  compute its expectation and variance
  under $P_{\nu}^{\tau}$. In particular, we will show that there is
  $c<\infty$ such that for every $t>0$,
  \begin{equation} \label{eq:expsumk}
    \bigg | E_{\nu}^{\tau}\Big[\sum_{k=1}^K U_N^k(t)\Big]
    - t \bigg| \le  c2^{-2\varepsilon_0N}, \qquad \mathbb P\text{-a.s.~as
      $N\to\infty$},
  \end{equation}
  and
  \begin{equation}
    \label{e:varest}
    \Var_{\nu}^{\tau}\bigg(\sum_{k=1}^K U_N^k(t)\bigg)
    \le c2^{-\frac 12 \varepsilon_0 N}, \qquad \mathbb P\text{-a.s.~as
      $N\to\infty$}.
  \end{equation}
  The statement of the proposition then follows from \lemref{lem:ltxk},
  \eqref{eq:expsumk} and \eqref{e:varest} by routine application of the
  Chebyshev inequality. Indeed, $\PP$-a.s.~for $N$ large enough,
  \begin{equation*}
    \begin{split}
      P^{\tau}_{\nu}&\bigg[\left| L_N(t) - t\right|
        \geq 2^{-\frac{1}{5}\varepsilon_0N}\bigg]\\
      &\leq P_{\nu}^{\tau}\bigg[L_N(t)\neq\sum_{k=1}^K U_N^k(t)\bigg]
      + P_{\nu}^{\tau}\bigg[\bigg|\sum_{k=1}^K U_N^k(t)-E_{\nu}^{\tau}
        \bigg[\sum_{k=1}^K U_N^k(t)\bigg]\bigg|
        \geq 2\cdot 2^{-\frac{1}{5}\varepsilon_0N}\bigg]\\
      &\leq c2^{-\frac{1}{2}\varepsilon_0N} + c' 2^{-\frac{1}{10}\varepsilon_0N}
      \leq c'' 2^{-\frac{1}{10}\varepsilon_0N},
    \end{split}
  \end{equation*}
  which is the claim of the proposition.

  We proceed by computing the expectation \eqref{eq:expsumk}. We will need two lemmas
  which we show later. The first lemma estimates the probability that a deep
  trap is visited by the process $Y$.

  \begin{lemma} \label{lem:approxhit}
    For every $t_N$ such that $1\leq t_N\leq2^N$, for every $\varepsilon>0$,
		$\PP$-a.s.~for $N$ large enough, for
    all $x\in\cD_N$,
    \begin{equation*}
      P^{\tau}_{\nu}[H_x\leq t_N ]
      = \frac{t_N} {E^{\tau}_{\nu}[H_x]} + O\big(t_N^2 2^{2(\varepsilon-1)N}\big)
      +O\big(2^{(\varepsilon-1)N}\big)\\
      \le c t_N 2^{(\varepsilon-1)N}.
    \end{equation*}
  \end{lemma}

  The second lemma then gives the expected contribution of a single
	$\ell_{t,x}^k$ to $\sum_{k=1}^K U_N^k(t)$.

  \begin{lemma} \label{lem:explocalalpha}
    For every fixed $t>0$, $k=1,\dots,K$ and $\varepsilon>0$, $\PP$-a.s.~for $N$ large enough, for
    all $x\in\cD_N$,
    \begin{equation*}
      E_{\nu}^{\tau}\big[\ell_{t,x}^k\big]
      = \frac{tR_N}{KE_{\nu}^{\tau}[H_x]}
      E_{x}^{\tau}\big[\ell_{\tx}(x)^{\alpha}\big]
      + O\big(2^{(2\gamma+3\varepsilon-2\varepsilon_0-2)N}\big).
	\end{equation*}
  \end{lemma}

  With Lemma~\ref{lem:explocalalpha} it is easy to compute the
  expectation \eqref{eq:expsumk}. Using that
  $|\cD_N| \leq c2^{(1-\gamma')N}$ by \eqref{eq:sizetop},
  and the definition \eqref{def:rN} of $R_N$,
  for every $\varepsilon>0$, $\PP$-a.s.~for $N$ large enough,
  \begin{align*}
    E^{\tau}_{\nu}\bigg[\sum_{k=1}^K U_N^k(t)\bigg]
    &= 2^{(\gamma'-\gamma) N}\sum_{x\in \cD_N}\sum_{k=1}^K
    \left(\frac{tR_N }{KE^{\tau}_{\nu}[H_x]}E^{\tau}_x[\ell_{\tx}(x)^{\alpha}]
      + O\big(2^{(2\gamma+3\varepsilon-2\varepsilon_0-2) N}\big)\right)\\
    &= t + O\big(2^{(\gamma'-\gamma)N}2^{(1-\gamma')N}
      2^{(2\gamma-2+3\varepsilon-\varepsilon_0) N}\big)\\
    &= t + O\big(2^{(\gamma-1+3\varepsilon-\varepsilon_0) N}\big).
  \end{align*}
  Choosing $\varepsilon<\varepsilon_0/3$ and recalling \eqref{eq:epsilon0}
  implies \eqref{eq:expsumk}.

  Next, we estimate the variance \eqref{e:varest}. Since $\nu $ is the stationary measure for
  $Y$, the random variables $U_N^k(t)$, $k=1,\dots,K$, are identically
  distributed under $P_{\nu}^{\tau}$. Hence
  \begin{equation} \label{eq:varsumL}
    \Var_{\nu}^{\tau}\bigg(\sum_{k=1}^K U_N^k(t)\bigg)
    = K\Var_{\nu}^{\tau}\big(U_N^1(t)\big)
    + 2 \sum_{1\le k< j\le K}\Cov_{\nu}^{\tau}\big(U_N^k(t),U_N^j(t)\big).
  \end{equation}

  The covariances can be  neglected easily. Indeed, since by definition
  $U_N^k(t)$ depends on the trajectory of $Y$ between times $t_{k-1}$ and
  $t_{k}-N^2 m_N$ only, we can use the Markov property at the later time to write
  \begin{equation}
    \label{e:cova}
    \Cov_\nu^\tau \big(U_N^k(t),U_N^j(t)\big) =
    E_\nu^\tau \big[\big(U^k_N(t)-E_\nu^\tau U^k_N(t)\big)
      E^\tau [U^j_N(t)- E_\nu^\tau  U^j_N(t)|Y_{t_k-N^2 m_N}]\big].
  \end{equation}
  By Lemma~\ref{l:tmixprop},
  $\big|P^\tau [Y_{t_k}=y|Y_{t_k-N^2 m_N}]-\nu_y\big| \le e^{-c N^2}$. Using in
  addition that $U^N_j\le e^{c'N}$ for some sufficiently large $c'$, we see
  that the inner expectation satisfies
  \begin{equation*}
    \big|E^\tau [U^j_N(t)- E_\nu^\tau  U^j_N(t)|Y_{t_k-N^2 m_N}]\big|
    \le e^{-c N^2 /2}.
  \end{equation*}
  Inserting this inequality back to \eqref{e:cova} and summing over $k<j$
  then implies that the second term in \eqref{eq:varsumL} is
  $O(e^{-cN^2})$ and thus can be neglected when proving \eqref{e:varest}.

  To control the variance of $U_N^1(t)$ in \eqref{eq:varsumL}, it is enough to bound its second
  moment, which is
  \begin{equation*}
    E_{\nu}^{\tau}\big[U_N^1(t)^2\big]
    = 2^{2(\gamma'-\gamma)N}\left(\sum_{x\in\cD_N}
      E_{\nu}^{\tau}\big[(\ell_{t,x}^1)^2\big] + \sum_{x\neq y \in \cD_N}
      E_{\nu}^{\tau}[\ell_{t,x}^1 \ell_{t,y}^1]\right).
  \end{equation*}
  Since, by definition, $\ell^1_{t,x}\le N^2 m_N$ and $\ell^1_{t,x}\neq 0$
  implies $H_x\le tR_N/K$,
  \begin{align}
    E_{\nu}^{\tau}\big[U_N^1(t)^2\big]
    &\leq 2^{2(\gamma'-\gamma)N}N^{4\alpha}m_N^{2\alpha}
    \Bigg(\sum_{x\in\cD_N}P_{\nu}^{\tau}\bigg[H_x\leq \frac{tR_N}{K}\bigg]
      + \sum_{x\neq y \in \cD_N} P_{\nu}^{\tau}
      \bigg[H_x,H_y\leq \frac{tR_N}{K}\bigg]\Bigg).
    \label{eq:varL12}
  \end{align}

  By \lemref{lem:approxhit} and \lemref{lem:boundrn},
  for every $\varepsilon>0$, $\PP$-a.s.~as $N\to\infty$,
  \begin{equation}\label{eq:hitK}
    P_{\nu}^{\tau}\bigg[H_x\leq \frac{tR_N}{K}\bigg]
    \le c2^{(\gamma-1+\varepsilon-\varepsilon_0)N}.
  \end{equation}
  Moreover, by \eqref{eq:sizetop}, $|\cD_N| \le c2^{(1-\gamma')N}$, and
  by \eqref{e:mN}, $N^{4\alpha}m_N^{2\alpha}\leq 2^{\varepsilon N}$, for every
  $\varepsilon>0$ and $N$
  large enough. It follows that the contribution of the first sum in
  \eqref{eq:varL12} to the variance, including the prefactor $K=2^{\varepsilon_0N}$ from
  \eqref{eq:varsumL}, can be bounded by
  \begin{equation*}
    c2^{(2(\gamma'-\gamma) +1-\gamma' + \gamma-1+2\varepsilon )N}
    = c 2^{(\gamma'-\gamma+2\varepsilon)N}.
  \end{equation*}
  By \eqref{eq:gamma1},
  $\gamma'-\gamma+2\varepsilon\leq-\varepsilon_0+2\varepsilon<-\frac{1}{2}\varepsilon_0$
  for $\varepsilon<\varepsilon_0/4$, and hence this contribution is
  smaller than
  $c2^{-\frac 12 \varepsilon_0 N}$ as required for \eqref{e:varest}.

  For the second summation in \eqref{eq:varL12} we write
  \begin{equation*}
    P_{\nu}^{\tau}\bigg[H_x,H_y\leq \frac{tR_N}{K}\bigg]
    \leq P_{\nu}^{\tau}\bigg[H_x < H_y\leq \frac{tR_N}{K}\bigg]
    + P_{\nu}^{\tau}\bigg[H_y < H_x \leq \frac{tR_N}{K}\bigg].
  \end{equation*}
  By the Markov property, each of these two probabilities can be bounded by
  \begin{align*}
    P_{\nu}^{\tau}\bigg[H_x < H_y\leq \frac{tR_N}{K}\bigg]
    &= \int_0^{\frac{tR_N}{K}} P_{\nu}^{\tau}[H_x\in du]
    P_x^{\tau}\bigg[H_y< \frac{tR_N}{K}-u\bigg]\\
    &\leq \int_0^{\frac{tR_N}{K}} P_{\nu}^{\tau}[H_x\in du]
    \left(P_x^{\tau}[H_y\leq \tx]
      + P_{\nu}^{\tau}\bigg[H_y\leq \frac{tR_N}{K}\bigg] \right)\\
    &\leq P_{\nu}^{\tau}\bigg[H_x\leq \frac{tR_N}{K}\bigg]
    \left(P_x^{\tau}[H_y\leq \tx]
      + P_{\nu}^{\tau}\bigg[H_y\leq \frac{tR_N}{K}\bigg]\right).
  \end{align*}
  Using \eqref{eq:hitK} and \eqref{eq:sizetop} again,
  the second sum in \eqref{eq:varL12} is bounded by
  \begin{equation}
    \label{e:blabla}
    c 2^{(\gamma -1+\varepsilon -\varepsilon_0)N}
    \Big( 2^{2(1-\gamma ')N} 2^{(\gamma -1+\varepsilon -\varepsilon_0)N}
      + \sum_{x\neq y \in \mathcal D_N}
      P^\tau_x[H_y\leq\tx]\Big).
  \end{equation}
  The first term in the parentheses of \eqref{e:blabla} together with the
  prefactors  $K$ from \eqref{eq:varsumL} and
  $2^{2(\gamma'-\gamma)N}N^{4\alpha}m_N^{2\alpha} \leq 2^{(2(\gamma'-\gamma)+\varepsilon)N}$
  from \eqref{eq:varL12}, contributes to the variance by at most
  \begin{equation*}
    c2^{(\varepsilon_0+2(\gamma'-\gamma)+\varepsilon + 2(1-\gamma')
          + 2(\gamma -1 +\varepsilon-\varepsilon_0))N}
      = c2^{(3\varepsilon - \varepsilon_0)N}\le c2^{-\frac 12
        \varepsilon_0 N}
  \end{equation*}
  if $\varepsilon $ is small enough, as required by \eqref{e:varest}.

  For the second term in the parentheses of \eqref{e:blabla} we need the
  following lemma whose proof is again postponed.

  \begin{lemma} \label{lem:Wtx}
    Let $\cW^x_t=\sum_{y\in\cD_N,y\neq x} \Ind{H_y\leq t}$. Then for every
    $\varepsilon>0$, $\PP$-a.s.~for $N$ large enough, for every $x\in \cD_N$,
    \begin{equation*}
      E^{\tau}_x[\cW^x_{\tx}] \leq 2^{\varepsilon N}.
    \end{equation*}
  \end{lemma}

  Using \lemref{lem:Wtx}, and including all the prefactors as before, the
  contribution of the second term in \eqref{e:blabla} to the variance \eqref{eq:varsumL}
  is bounded by
  \begin{equation*}
    c2^{(\varepsilon_0 + 2(\gamma'-\gamma)+\epsilon + 1-\gamma'
          + \gamma-1+\varepsilon-\varepsilon_0 + \varepsilon )N}
      = c \, 2^{ (\gamma'-\gamma + 3\varepsilon )N}
      \le 2^{-\frac{1}{2}\varepsilon_0N},
  \end{equation*}
  where for the last inequality we used \eqref{eq:gamma1} again, and
  choose $\varepsilon $ small enough.
  This completes the proof of \eqref{e:varest} and thus of the
  proposition.
\end{proof}

We proceed by proving the lemmas used in the above proof.

\begin{proof}[Proof of \lemref{lem:approxhit}]
  By \cite[Theorem~1]{AB92} the hitting time $H_x$ is
  approximately exponential in the sense that
  \begin{equation*}
    \left|P^{\tau}_{\nu}[H_x>t]-e^{-\frac{t}{E^{\tau}_{\nu}[H_x]}}\right|
    \leq \frac{1}{\lambda_Y E^{\tau}_{\nu}[H_x]}.
  \end{equation*}
  Hence, using Propositions~\ref{prop:spectralgap} and \ref{prop:meanhit}
  to bound $\lambda_Y$ and $E^{\tau}_{\nu}[H_x]$ respectively, we
  have for every $\varepsilon>0$, $\PP$-a.s.~for $N$ large enough,
  \begin{equation*}\begin{split}
      P^{\tau}_{\nu}[H_x\leq t_N ] &= (1-e^{-\frac{t_N}{E^{\tau}_{\nu}[H_x]}})
      + O\big(2^{(\varepsilon-1)N}\big)\\
      &= \frac{t_N}{E^{\tau}_{\nu}[H_x]}
      + O\big(t_N^2 2^{2(\varepsilon-1)N}\big)
      + O\big(2^{(\varepsilon-1)N}\big).
  \end{split}\end{equation*}
  Finally, if $1\leq t_N\leq 2^N$ this is bounded by $ct_N2^{(\varepsilon-1)N}$,
  which proves the lemma.
\end{proof}

\begin{proof}[Proof of \lemref{lem:explocalalpha}]
  By the strong Markov property and the definition of $\ell^k_{t,x}$,
  \begin{equation}
    \label{eq:expltxk}
    \begin{split}
      E_{\nu}^{\tau}\big[\ell_{t,x}^k\big] &\ge
      P_{\nu}^{\tau}\big[H_x\in[t_{k-1},t_k-2N^2m_N]\big]
      E_x^{\tau}\big[\ell_{N^2m_N}(x)^{\alpha}\big],\\
      E_{\nu}^{\tau}\big[\ell_{t,x}^k\big]
      &\leq P_{\nu}^{\tau}\big[H_x\in[t_{k-1},t_k-N^2m_N]\big]
      E_x^{\tau}\big[\ell_{N^2m_N}(x)^{\alpha}\big].
  \end{split}\end{equation}
  We will now give approximations of the expressions appearing in
  \eqref{eq:expltxk}.

  Observe that for every $s,t>0$,
  \begin{equation*}
    \ell_{t}(x)^{\alpha} \leq \ell_{s}(x)^{\alpha}
    + (\ell_{t}(x)-\ell_s(x))^{\alpha}.
  \end{equation*}
  Using this inequality with $t=N^2m_N$ and $s=\tx$ and applying the strong Markov
  property at $\tx$, observing that
  $Y_\tx$ is $\nu $-distributed,
  \begin{equation*}
    E_x^{\tau}\big[\ell_{N^2m_N}(x)^{\alpha}\big]
    \leq E_x^{\tau}\big[\ell_{\tx}(x)^{\alpha}\big]
    + E_{\nu}^{\tau}\big[\ell_{N^2m_N}(x)^{\alpha}\big].
  \end{equation*}
  By \lemref{lem:approxhit}, using also that by \eqref{e:mN},
  $\ell_{N^2m_N}(x)^{\alpha}\leq N^{2\alpha}m_N^{\alpha} \leq 2^{\varepsilon N}$
  for every $\varepsilon>0$
  and $N$ large enough,
  \begin{equation*}
    E_{\nu}^{\tau}\big[\ell_{N^2m_N}(x)^{\alpha}\big]
    \leq P_{\nu}^{\tau}\big[H_x\leq N^2m_N] 2^{\varepsilon N}
    \leq c2^{(3\varepsilon-1)N}.
  \end{equation*}
  Hence we obtain the upper bound
  \begin{equation}\label{eq:blubbup}
    E_x^{\tau}\big[\ell_{N^2m_N}(x)^{\alpha}\big]
    \leq E_x^{\tau}\big[\ell_{\tx}(x)^{\alpha}\big]
    + c2^{(3\varepsilon-1)N}.
  \end{equation}
  For a matching lower bound, note that
  \begin{equation*}
    E_x^{\tau}\big[\ell_{N^2m_N}(x)^{\alpha}\big]
    \geq E_x^{\tau}\big[\ell_{\tx}(x)^{\alpha}
      \Ind{\tx\leq N^2m_N}\big].
  \end{equation*}
  But from Proposition~\ref{prop:mixtime} it follows that
  \begin{equation*}
    E^{\tau}_x[\ell_{\tx}(x)^{\alpha}\Ind{\tx>N^2m_N}]
    \leq E^{\tau}_x[\tx^{\alpha}\Ind{\tx>N^2m_N}]
    \leq \sum_{k=N^2}^\infty(km_N)^{\alpha}e^{-k} \leq c e^{-c'N^2},
  \end{equation*}
  so that
  \begin{equation}\label{eq:blubblow}
    E_x^{\tau}\big[\ell_{N^2m_N}(x)^{\alpha}\big]
    \geq E_x^{\tau}\big[\ell_{\tx}(x)^{\alpha}\big]
    - ce^{-cN^2}.
  \end{equation}
  Combining \eqref{eq:blubbup} and \eqref{eq:blubblow}, we obtain
  \begin{equation}\label{eq:ltinterval}
    E_x^{\tau}\big[\ell_{N^2m_N}(x)^{\alpha}\big]
    = E_{x}^{\tau}\big[\ell_{\tx}(x)^{\alpha}\big]
    + O\big(2^{(3\varepsilon-1)N}\big).
  \end{equation}
  Note also that by \eqref{e:mN}, for every $\varepsilon>0$ and $N$ large enough,
  \begin{equation}
    \label{e:yub}
    E^{\tau}_x[\ell_{\tx}(x)^{\alpha}]
    \leq E^{\tau}_x[\tx^{\alpha}]\leq cm_N^{\alpha} \leq 2^{\varepsilon N}.
  \end{equation}

  To approximate the probabilities in \eqref{eq:expltxk}, we apply
  \lemref{lem:approxhit} for $t_N=t_{k-1}$ and
  $t_N=t_k-iN^2m_N$, for a fixed $t>0$ and $i=1,2$. Using \lemref{lem:boundrn} to
  bound $R_N$ and \propref{prop:meanhit} to bound $E^{\tau}_{\nu}[H_x]$,
  for every $\varepsilon>0$, $\PP$-a.s.~for $N$ large enough, for both
  $i=1,2$,
  \begin{equation}\label{eq:hitinterval}
    P_{\nu}^{\tau}\big[H_x\in[t_{k-1},t_k-iN^2m_N]\big]
    =	\frac{tR_N}{KE_{\nu}^{\tau}[H_x]}
    + O\big(2^{2(\gamma+\varepsilon-\varepsilon_0-1)N}\big)
    = O\big(2^{(\gamma+\varepsilon-\varepsilon_0-1)N}\big).
  \end{equation}

  Inserting both \eqref{eq:hitinterval} and \eqref{eq:ltinterval} in
  \eqref{eq:expltxk}, and using \eqref{e:yub}, for every $\varepsilon>0$,
  $\PP$-a.s.~for $N$ large enough,
  \begin{equation*}
    E_{\nu}^{\tau}\big[\ell_{t,x}^k\big]
    = \frac{tR_N}{KE_{\nu}^{\tau}[H_x]}
    E_{x}^{\tau}\big[\ell_{\tx}(x)^{\alpha}\big]
    + O\big(2^{(2\gamma+3\varepsilon-2\varepsilon_0-2)N}\big).
  \end{equation*}
  This proves the lemma.
\end{proof}

\begin{proof}[Proof of \lemref{lem:ltxk}]
  Note first that
  \begin{equation*}
    \bigg\{L_N(t)\neq\sum_{k=1}^K U_N^k(t)\bigg\} \subseteq
    \bigg\{\exists x\in\cD_N:~\ell_{tR_N}(x)^{\alpha}\neq
      \sum_{k=1}^K \ell_{t,x}^k\bigg\}.
  \end{equation*}
  To control the probability of this event, we introduce some more notation.
  Set $H_x^{(0)}=0$, $H_x^{(1)}=H_x$, and for $k\ge 2$ define the time of the
	`$k$-th visit after mixing' inductively as
  \begin{equation*}
    H_x^{(k)}=\inf\{t>\tx\circ\theta_{H_x^{(k-1)}}+H_x^{(k-1)}:~Y_t=x\}.
  \end{equation*}
  Let $\cN_t^x=\min\{k\geq0,~H_x^{(k)}\leq t\}$ be the number of `visits after
  mixing' to $x$ before time $t$. Finally, let
  $I_k=[t_k-2N^2m_N, t_k]$. Then
  \begin{equation}
    \label{e:520}
    \begin{split}
      P^{\tau}_{\nu}\bigg[\exists x\in\cD_N:~\ell_{tR_N}(x)^{\alpha}\neq
        \sum_{k=1}^K \ell_{t,x}^k\bigg]
      &\leq P^{\tau}_{\nu}\bigg[Y_s\in \mathcal D_N
        \text{ for some }s\in \bigcup_{k=1}^K I_k\bigg]
      \\
      &\quad+P^{\tau}_{\nu}\big[\exists x\in\cD_N:~\cN_{tR_N}^x\geq2\big]
      \\
      &\quad+P^{\tau}_{\nu}\big[\exists x\in\cD_N:~\tx\circ\theta_{H_x}>
				N^2m_N\big].
    \end{split}
  \end{equation}
  We show that each of the three terms on the right-hand side is smaller than
  $c2^{-\frac{1}{2}\varepsilon_0 N}$, which will prove the lemma.

  For the first term in \eqref{e:520}, using the stationarity of $\nu $
  and the Markov property
  \begin{equation}
    \label{e:iuy}
     P^{\tau}_{\nu}\bigg[Y_s\in \mathcal D_N \text{ for some }
       s\in  \bigcup_{k=1}^K I_k\bigg]
     \le  K \sum_{x\in \mathcal D_N}
     P^{\tau}_{\nu}\big[H_x \le 2 N^2 m_N\big].
  \end{equation}
  By \lemref{lem:approxhit}, $\PP$-a.s.~for all $x\in \cD_N$, for
  $\varepsilon>0$ small and $N$ large enough,
  \begin{equation*}
    P^{\tau}_{\nu}[H_x \leq 2 N^2m_N] \leq 2^{(\varepsilon-1)N}.
  \end{equation*}
  Since
  $|\cD_N| \le c2^{(1-\gamma ')N}$ by \eqref{eq:sizetop},
  the right hand side of \eqref{e:iuy} is bounded by
  $c 2^{\varepsilon_0 N} 2^{(\varepsilon-\gamma')N}$.
  Since $\gamma'>1/2$ and by definition $\varepsilon_0\leq 1/4$, when
  $\varepsilon $ is small enough this is smaller
  than $c 2^{-\frac12 \varepsilon_0 N}$ as required.

  For the second term in \eqref{e:520}, by \lemref{lem:approxhit} and
  the strong Markov property at $\tx$, for every $\varepsilon>0$, $\PP$-a.s.~for
  $N$ large enough,
  \begin{equation*}
    P^{\tau}_{\nu}[H_x^{(2)} \leq tR_N]
    \leq P^{\tau}_{\nu}[H_x\leq tR_N]^2 \leq c2^{2(\gamma-1+\varepsilon)N}.
  \end{equation*}
  Together with \eqref{eq:sizetop} to bound $|\cD_N|$, and using
  \eqref{eq:epsilon0} and \eqref{eq:gamma1}, $\PP$-a.s.~for $N$ large enough,
  \begin{align*}
    P^{\tau}_{\nu}\big[\exists x\in\cD_N:~\cN_{tR_N}^x\geq2\big]
    &\leq c2^{(1-\gamma')N} 2^{2(\gamma-1+ \varepsilon)N}\\
    &= c2^{(\gamma-\gamma')N + (\gamma-1)N + \varepsilon N}\\
    &\leq c 2^{(-\varepsilon_0+\varepsilon) N} \le c2^{-\frac 12
      \varepsilon_0N}
  \end{align*}
  as required.

  Finally we give a bound on the third term in \eqref{e:520}. By
	\propref{prop:mixtime},
  $P^{\tau}_x[\tx> N^2m_N] \leq e^{-cN^2 }$. Thus, with \eqref{eq:sizetop}
  to bound $|\cD_N|$, $\PP$-a.s.~for $N$ large enough,
  \begin{align*}
    P^{\tau}_{\nu}\big[\exists x\in\cD_N:~\tx\circ\theta_{H_x}> N^2m_N\big]
    &\leq c 2^{(1-\gamma')N} P^{\tau}_x[\tx> N^2m_N]\\
    &\leq c'2^{-\frac{1}{2}\varepsilon_0N}.
  \end{align*}
  Together with the previous estimates, this implies that the right-hand side
	of \eqref{e:520} is bounded by $c2^{-\frac{1}{2}\varepsilon_0N}$, and
  concludes the proof of the lemma.
\end{proof}

\begin{proof}[Proof of \lemref{lem:Wtx}]

  Let $\mathcal H_0=0$ and define recursively for $i\ge 1$
  \begin{equation*}
    \mathcal H_i = \inf\{t\ge \mathcal H_{i-1}: Y_t\in \mathcal
      D_N\setminus\{Y_{\mathcal H_{i-1}}\}\}.
  \end{equation*}
  By \eqref{eq:disttop}, $\PP$-a.s.~for $N$ large enough, the vertices in
  $\cD_N$ are at least distance $\delta N$ from each other. In particular
  the balls $A_x = B(x,\rho_x)$, $x\in \mathcal D_N$, constructed in
  \lemref{lem:spheres} are disjoint. Hence, when on $y\in \mathcal D_N$,
  the random walk $Y$ should first leave $A_y$ in order to visit
  $\mathcal D_N\setminus\{y\}$.
  The strong Markov property and
  Corollary~\ref{cor:green} then imply that $\mathcal H_i$ stochastically
  dominates a Gamma random variable with parameters $i$ and
  $\mu :=c2^{N^{1-\delta }}$.

  If $\mathcal W^x_t\ge i$, then $\mathcal H_i\le t$. Hence, for
  $t\ge \mu $,
  \begin{equation*}
    E^\tau_x \big[\mathcal W^x_t\big]
    =\sum_{i\ge 1} P^\tau_x[\mathcal W^x_t\ge i]
    \le\sum_{i\ge 1} P^\tau_x[\mathcal H_i\le t]
    \le \sum_{i\ge 1}
    \int_0^t \mu^i u^{i-1} e^{-\mu u }\Gamma (i)^{-1} du
    = \mu  t.
  \end{equation*}
  It follows that
  \begin{equation*}
    E^\tau_x \big[\mathcal W_\tx^x\big] \le E^\tau_x \big[\mathcal W_{N^2 m_N}^x\big] +
    |\mathcal D_N| P^\tau_x[\tx \ge N^2 m_N]
    \le \mu N^2 m_N  + c 2^{(\gamma '-1)N}e^{-cN^2}  \le 2^{\varepsilon N}
  \end{equation*}
  by \eqref{eq:sizetop}, \eqref{e:mN} and Proposition~\ref{prop:mixtime}.
  This completes the proof.
\end{proof}

For later applications, we state two further consequences of the proof of
\lemref{lem:ltxk}.

\begin{lemma}\label{lem:nolargelocaltimes}
  $\PP$-a.s.~for $N$ large enough,
  \begin{equation*}
    P^{\tau}_{\nu}\big[\exists x\in\cD_N:~\ell_{t R_N}(x)> N^2m_N\big]
    \leq c2^{-\frac{1}{2}\varepsilon_0N},
  \end{equation*}
  and
  \begin{equation*}
    P^{\tau}_{\nu}\Big[\big|\{x\in\cD_N:~H_x\leq tR_N\}\big|
      \geq 2^{\frac{3}{2}\varepsilon_0N}\Big]
    \leq c2^{-\frac{1}{4}\varepsilon_0 N}.
  \end{equation*}
\end{lemma}

\begin{proof}
  The first claim follows directly from the bounds on the second and third
	term on the right hand side of \eqref{e:520} in the proof of
	\lemref{lem:ltxk}, since the local time in a vertex that is only `visited
	once after mixing' is bounded by $\tx\circ\theta_{H_x}$.

  The second assertion can be seen in the following way.
  Using \lemref{lem:approxhit} to bound the probability of a single vertex
  $x\in \cD_N $ to be visited before time $tR_N$ and \eqref{eq:sizetop} to
  bound the size of $\cD_N$, for every $\varepsilon>0$, $\PP$-a.s.~for $N$ large
  enough,
  \begin{equation*}
    E^{\tau}_{\nu}\big[|\{x\in\cD_N:~H_x\leq tR_N\}|\big]
    \leq c2^{(1-\gamma')N}2^{(\gamma-1+\varepsilon)N} \leq
    c2^{(\gamma-\gamma'+\varepsilon)N}.
  \end{equation*}
  By \eqref{eq:gamma1} this is equal to
  $c2^{(\varepsilon_0+\varepsilon)N}$, so choosing $\varepsilon<\varepsilon_0/4$
  this is smaller than $c2^{\frac{5}{4}\varepsilon_0 N}$. Then by the Markov
  inequality the probability that there are more than
  $2^{\frac{3}{2}\varepsilon_0 N}$ vertices visited is smaller than
  $c2^{-\frac{1}{4}\varepsilon_0 N}$.
\end{proof}


\section{Clock process of the deep traps} 
\label{sec:deep}

This section contains the main steps leading to the proof of \thmref{thm:main}.
Recall from \eqref{e:SD} that the `clock process of deep traps'
$S_\mathcal D$ is given by
\begin{equation*}
  S_{\cD}(t) = \int_0^{t} (1\vee\tau_{Y_s})\Ind{Y_s\in \cD_N}ds =
  \int_0^{t} \tau_{Y_s}\Ind{Y_s\in \cD_N}ds.
\end{equation*}
We now show that $S_\mathcal D$ converges to a stable process.

\begin{proposition} \label{prop:deep}
Under the assumptions of \thmref{thm:main}, the rescaled clock processes of
the deep traps $g_N^{-1}S_{\cD}(tR_N)$ converge in $\PP$-probability as
$N\to\infty$, in $P^{\tau}_{\nu}$ distribution on the space $D([0,T],\RR)$
equipped	with the Skorohod $M_1$-topology, to an $\alpha$-stable
subordinator $V_{\alpha}$.
\end{proposition}

The proof of \propref{prop:deep} consists of three steps. In a first step,
we show convergence in distribution of one-dimensional marginals by
showing that the Laplace transform of one-dimensional marginals
converges. This step contains, to some extent, the principal insight of this
paper and is split in two parts: We first show
the quasi-annealed convergence mentioned in the introduction, which is
then strengthened to convergence in probability with respect to the
environment. The second and third step of
the proof of Proposition~\ref{prop:deep} are rather standard  and
deal with the joint convergence of increments and the tightness.


\subsection{Quasi-annealed convergence}

We establish here the connection between the Lap\-lace transform of the
clock process of deep traps and the local time functional $L_N$ studied in
Section~\ref{sec:localtimes}. The key observation is that the depths of
the deep traps are in some sense independent of the fast chain $Y$,
and can be thus averaged out easily.

To formalize this, we introduce a two-step procedure to sample the environment
$\tau$. Let $\xi=(\xi_x)_{x\in\HH_N}$ be i.i.d.~Bernoulli random variables
such that, cf.~\eqref{eq:density},
\begin{equation*}
  \PP[\xi_x=1]=1-\PP[\xi_k=0]=\PP[x\in\cD_N]= 2^{-\gamma 'N}(1+o(1)).
  \end{equation*}
Further, let
$\overline{E}=(\overline{E}_x)_{x\in\HH_N}$ be i.i.d.~standard Gaussian
random variables
conditioned to be larger than $\frac{1}{\beta\sqrt{N}}g'_N$, and
$\underline{E}=(\underline{E}_x)_{x\in\HH_N}$ i.i.d.~standard Gaussian
random variables
conditioned to be smaller than $\frac{1}{\beta\sqrt{N}}g'_N$. The collections
$\xi$, $\overline{E}$ and $\underline{E}$ are mutually independent.
The Hamiltonian of the REM can be obtained
by setting
\begin{equation}
  \label{e:Experc}
  E_x=\overline{E}_x \Ind{\xi_x=1} + \underline{E}_x \Ind{\xi_x=0}.
\end{equation}
From now on, we always assume that $E_x$ are given by \eqref{e:Experc}.
Observe that in this procedure the set $\mathcal D_N$ coincides with the set
$\{x\in \mathbb H_N:\xi_x=1\}$.

We use $\cG=\sigma(\xi,\underline{E})$ to denote the $\sigma $-algebra generated by
the $\xi $'s and $\underline E$'s. In particular, the
number and positions of deep traps and all the $\tau_y$, $y\notin\cD_N$, are
$\cG$-measurable. The depths of deep traps are however
independent of $\cG$.

In the next lemma we compute the quasi-annealed Laplace transform of
$S_{\mathcal D}$. The term `quasi-annealed' refers to the fact that we
average over the energies of the deep traps $\overline{E}_x$ (and over the
  law of the process), but we keep quenched the positions of the deep
traps $\xi_x$ and the energies of remaining traps $\underline{E}_x$.

\begin{lemma}
  \label{lem:annealed}
  There is a constant $\mathcal K\in (0,\infty)$ such that
  for every $\lambda >0$ and $t\ge 0$,
  \begin{equation*}
    \EE\Big[E^{\tau}_{\nu}\big[
        e^{-\frac{\lambda}{g_N}S_{\cD}(tR_N)}\big] \Big|\,\cG \Big]
    \xrightarrow{N\to\infty} e^{-\cK\lambda^{\alpha} t}, \qquad \mathbb P\text{-a.s.}
  \end{equation*}
\end{lemma}

\begin{proof}
  Recall the separation event $\mathscr{S}$ defined in \eqref{eq:disttop}.
  This event depends only on $\xi$ and is therefore $\cG$-measurable, and
  by \lemref{lem:sparsetop} it occurs $\PP$-a.s.~for $N$ large enough.
  On $\mathscr{S}$, no deep traps $x\in\cD_N$ are neighbors. Since
  moreover $\tau_x\geq1$ for $x\in\cD_N$, all the transition rates
  \begin{equation*}
    q_{xy} \mathbf{1}_{\mathscr{S}}
    = \frac{\tau_x\wedge\tau_y}{1\wedge\tau_x}\mathbf{1}_{\mathscr{S}}
    ,\qquad x,y\in \mathbb H_N,
  \end{equation*}
  are $\cG$-measurable. That is, on the event $\mathscr{S}$, the law of
  the chain $Y$ is
  in fact $\cG$-measurable. Therefore, on $\mathscr{S}$, the order of taking
  expectations over the depth of the deep traps and the chain $Y$ can be
  exchanged. Namely, denoting by $\overE$ the expectation over the random
  variables $\overline{E}_x$, on $\mathscr{S}$,
  \begin{equation}	\label{eq:exchange}
    \begin{split}
      \EE\left[E^{\tau}_{\nu}\big[
          e^{-\frac{\lambda}{g_N}S_{\cD}(tR_N)}\big]\, \Big |\, \cG  \right]
      &=E^{\tau}_{\nu}\left[\overE\left[
          e^{-\frac{\lambda}{g_N}S_{\cD}(tR_N)}\right] \right]
      \\&=E^{\tau}_{\nu}\left[\overE\Big[
          \exp\Big\{-\frac{\lambda}{g_N}
            \int_0^{tR_N}\tau_{Y_s}\Ind{Y_s\in \mathcal D_N}ds\Big\}\Big] \right]
      \\&=E^{\tau}_{\nu}\left[\overE\Big[
          \exp\Big\{-\frac{\lambda}{g_N}
            \sum_{x\in \mathcal D_N}\ell_{tR_N}(x)\tau_x\Big\}\Big] \right].
    \end{split}
  \end{equation}

  We next approximate the inner expectation on the right-hand side of
  \eqref{eq:exchange}. Since its argument is bounded by one, it will be
  sufficient to control it on an event of $P^{\tau}_{\nu}$-probability
  tending to $1$ as $N\to\infty$. Define the event
  \begin{equation}\label{eq:eventA}
    \cA = \big\{\text{for all }x\in\cD_N,~\ell_{tR_N}(x) \leq N^2 m_N\big\}\\
    \cap \Big\{\big|L_N(t) - t\big| \leq 2^{-\frac{1}{5}\varepsilon_0N}\Big\}.
  \end{equation}
	By \propref{prop:conclocal} and \lemref{lem:nolargelocaltimes},
  $\PP$-a.s.~for $N$ large enough, $P^{\tau}_{\nu}[\cA^c]\leq e^{-cN}$.

  When performing the inner expectation of \eqref{eq:exchange}, the local
  times $\ell_{tR_N}(x)$ of $Y$ as well as $\cD_N$ are fixed, the
  expectation is taken only over the energies of the deep traps. By
  independence of the $\overline{E}_x$ it follows that
  \begin{equation}\label{eq:lapcond}\begin{split}
      \overE\left[
        e^{-\frac{\lambda}{g_N}S_{\cD}(tR_N)}\right]
      &= \prod_{x\in \cD_N}
      \overE\left[\exp\left\{- \frac{\lambda}{g_N}\ell_{tR_N}(x)
          e^{\beta\sqrt N \,\overline{E}_x}\right\}\right]\\
      &= \exp\left\{\sum_{x\in \cD_N} \log \overE\left[
          \exp\left\{- \frac{\lambda }{g_N}\ell_{tR_N}(x)
            e^{\beta\sqrt N \,\overline{E}_x} \right\}\right]\right\}.
  \end{split}\end{equation}

  For $u\in [0,N^2 m_N]$, let
  \begin{equation*}
    \vartheta(u) = 1- \overE\left[\exp\left\{-\frac{\lambda}{g_N}
        u e^{\beta\sqrt N \,\overline{E}_x}\right\}  \right].
  \end{equation*}
  Since $(\overline{E}_x)$ has standard Gaussian distribution conditioned
  on being larger than $\frac{1}{\beta\sqrt{N}}\log g'_N$, using that by
  \eqref{eq:density},
  \begin{equation*}
    \PP\Big[E_x > \frac{1}{\beta\sqrt{N}}\log g'_N\Big] = \PP[x\in\cD_N]
    = 2^{-\gamma' N}(1+o(1)),
  \end{equation*}
  it follows that
  \begin{equation*}
    \vartheta(u) = \frac{2^{\gamma' N}}{\sqrt{2\pi}}\, (1+o(1))
    \int_{\frac{1}{\beta\sqrt{N}}\log g'_N}^{\infty} e^{-\frac{s^2}{2}}
    \left(1-e^{-\frac{\lambda u}{g_N}e^{\beta\sqrt N s}} \right) ds.
  \end{equation*}
  We use the substitution $s=\frac{1}{\beta\sqrt{N}} ( \beta z + \log g_N - \log
    \lambda-\log u)$. The lower limit of the integral then becomes
  \begin{equation*}
    \frac{1}{\beta}(\log g'_N-\log g_N+\log\lambda+\log u)
    =: \omega(N).
  \end{equation*}
  For $u\le N^2 m_N$, $\omega_N$ is asymptotically dominated by
  $\log g'_N-\log g_N \leq -cN$, and thus $\lim_{N\to\infty}\omega(N)=-\infty$.
  After the substitution,
  \begin{equation}
    \label{e:thetaa}
    \vartheta(u)=\frac{2^{\gamma' N}}{\sqrt{2\pi}}  (1+o(1))
    \int_{\omega(N)}^{\infty}e^{-\frac{1}{2\beta^2 N}
      (\beta z + \log g_N - \log \lambda- \log u)^2}
    \left(1-e^{-e^{\beta z}} \right) \frac{1}{\sqrt{N}}\, dz.
  \end{equation}

  For $u\in [0,N^2 m_N]$, using the definition \eqref{def:gn} of $g_N$,
  the exponent of the first exponential satisfies
  \begin{equation}\label{eq:squarexp}\begin{split}
    -\frac{1}{2\beta^2 N}&(\beta z+\log g_N-\log\lambda-\log u)^2 \\
    &= -\frac{1}{2\beta^2 N}(\beta z + \alpha\beta^2 N -\frac{1}{\alpha}
      \log(\alpha\beta\sqrt{2\pi N}) - \log \lambda - \log u )^2 \\
    &= -\frac{\alpha^2\beta^2}{2}N + \alpha\log\lambda + \alpha \log u
    +\log(\alpha\beta\sqrt{2\pi N})
    - \alpha\beta z + \err(z)+ o(1).
  \end{split}\end{equation}
  Here, $o(1)$ is an error independent of the variable $z$. Note that for the
	$\log^2 u$ part to be $o(1)$ it is important that $m_N$ defined in \eqref{e:mN}
	is not too large, see also \remref{rem:convmode}. The second error term is
  \begin{equation*}
    \err(z)=-\frac{1}{2N} z^2 + \frac{1}{\beta N} z \bigg(
      \frac{1}{\alpha}\log(\alpha\beta\sqrt{2\pi N}) + \log \lambda
      + \log u\bigg).
  \end{equation*}
  Observe that  $\lim_{N\to\infty} \err(z)=0$ for every $z\in \mathbb R$,
  and that for every $\varepsilon $ there is $N_0$ large enough, so that
  for $N\ge N_0$ and all $z\in \mathbb R$
  \begin{equation}
    \label{e:errbound}
   \err(z) \leq \varepsilon |z|.
  \end{equation}

  Inserting the results of the computation \eqref{eq:squarexp} back into
  \eqref{e:thetaa}, using that ${\alpha^2\beta^2}/{2}=\gamma\log2$, we
  obtain
  \begin{equation}
    \label{eq:thetamid}
    \vartheta(u) = \alpha\beta  2^{(\gamma'-\gamma) N}
    \lambda^{\alpha}u^{\alpha}
    \int_{\omega(N)}^{\infty} e^{-\alpha\beta z + \err(z)}
    \left(1-e^{-e^{\beta z}} \right) dz\, (1+o(1)).
  \end{equation}

  We now claim that
  \begin{equation}
    \label{eq:convint}
    \int_{\omega(N)}^{\infty} e^{-\alpha\beta z + \err(z)}
    \left(1-e^{-e^{\beta z}}\right)dz
    \xrightarrow{N\to\infty}
    \int_{\mathbb R} e^{-\alpha\beta z }
    \left(1-e^{-e^{\beta z}}\right)dz=:C.
  \end{equation}
  Indeed, the
  integrand converges point-wise on $\RR$ to
  $e^{-\alpha\beta z}(1-e^{-e^{\beta z}})$ which is integrable if
  $\alpha<1$. Moreover, by \eqref{e:errbound},
  the integrand is bounded by $e^{-\alpha\beta z+\varepsilon |z|}(1-e^{-e^{\beta z}})$,
  which is integrable
  if we choose $\varepsilon<\beta(1-\alpha)\wedge\alpha\beta$. The claim
  \eqref{eq:convint} follows by the dominated convergence theorem.

  We now come back to \eqref{eq:lapcond}. Since on $\mathcal A$,
  $\ell_{tR_N}(x)\le N^2 m_N$ for all $x\in \mathcal D_N$, and
  $\gamma '<\gamma $, we see that $\vartheta(\ell_{tR_N}(x)) = o(1)$
  uniformly in $x\in \mathcal D_N$ on $\cA$.
  With $\log(1-x)= -x(1+O(x))$ as $x\to0$ this yields
  \begin{equation*}
    \begin{split}
      \overE\left[
        e^{-\frac{\lambda}{g_N}S_{\cD}(tR_N)}\right]
      &= \exp\Big\{\sum_{x\in \mathcal
          D_N}\log \big(1-\vartheta(\ell_{tR_N}(x))\big)\Big\}
      \\&= \exp\Big\{-\sum_{x\in \mathcal
          D_N}\vartheta(\ell_{tR_N}(x))(1+o(1))\Big\}.
    \end{split}
  \end{equation*}
  The inner sum can be easily computed from \eqref{eq:thetamid}. Recalling
  that on $\mathcal A$ the
  local time functional $L_N(t)$ converges,
  denoting $\cK=\alpha\beta C$, we obtain on $\cA$,
  \begin{equation}
    \label{eq:theta}
    \begin{split}
      \sum_{x\in \cD_N} \vartheta(\ell_{tR_N}(x)) &= \alpha\beta C \lambda^{\alpha}
      2^{(\gamma'-\gamma) N}\sum_{x\in \cD_N} \ell_{tR_N}(x)^{\alpha}
      (1+o(1)) \\
      &=\alpha\beta C \lambda^{\alpha} L_N(t) (1+o(1))\\
      &= \cK \lambda^{\alpha} t +o(1) \quad \text{as }N\to\infty.
    \end{split}
  \end{equation}
  It follows that on $\mathcal A$
  \begin{equation*}
    \overE\Big[e^{-\frac{\lambda}{g_N}S_{\cD}(tR_N)}\Big]
    = e^{-\cK t \lambda^{\alpha}(1+o(1))}
    = e^{-\cK t \lambda^{\alpha}} + o(1)\quad \text{as }N\to\infty.
  \end{equation*}
  Inserting this into \eqref{eq:exchange}, using that
  $P^{\tau}_{\nu}[\cA^c]=O(e^{-cN})$, we conclude that, on $\mathscr{S}$,
  $\PP$-a.s.~as $N\to\infty$,
  \begin{equation*}
    \EE\left[E^{\tau}_{\nu}\left[
        e^{-\frac{\lambda}{g_N}S_{\cD}(tR_N)}\right]\,\Big | \,\cG \right]
    = E^{\tau}_{\nu}\left[\overE\left[e^{-\frac{\lambda}{g_N}S_{\cD}(tR_N)}
        \right] \mathbf{1}_{\cA}\right] + O(e^{-cN})\\
    = e^{-\cK t \lambda^{\alpha}} + o(1).
  \end{equation*}
  Since $\mathscr{S}$ occurs $\PP$-a.s.~for $N$ large enough, this completes
the proof.
\end{proof}


\subsection{Quenched convergence}

We strengthen the convergence in \lemref{lem:annealed} in the following way.

\begin{lemma} \label{lem:onedim}
The one-dimensional marginals of the rescaled clock processes
$g_N^{-1}S_{\cD}(tR_N)$ converge in $\PP$-probability as $N\to\infty$, in
$P^{\tau}_{\nu}$-distribution to an $\alpha$-stable law, that is for every
$t>0$ and $\lambda>0$,
\begin{equation*}
	E^{\tau}_{\nu}\left[e^{-\frac{\lambda}{g_N}S_{\cD}(tR_N)}\right]
		\xrightarrow{N\to\infty}
		e^{-\cK\lambda^{\alpha} t} \qquad \text{in }\PP\text{-probability.}
\end{equation*}
\end{lemma}

\begin{proof}
It will be enough to show that
$\PP$-a.s.~for $N$ large enough,
\begin{equation} \label{eq:quenched}
  \EE\Big[E^{\tau}_{\nu}\big[
      e^{-\frac{\lambda}{g_N}S_{\cD}(tR_N)}\big]^2 \Big |\,\cG \Big]
  = e^{-2\cK\lambda^{\alpha} t} +o(1).
\end{equation}
Indeed, if \eqref{eq:quenched} holds, then the conditional variance
\begin{equation*}
  \Var\Big[E^{\tau}_{\nu}\big[
      e^{-\frac{\lambda}{g_N}S_{\cD}(tR_N)}\big]\Big |\,\mathcal G \Big]
  \xrightarrow{N\to\infty} 0, \qquad \mathbb P\text{-a.s.},
\end{equation*}
and the claim follows by an application of the Chebyshev inequality and
Lemma~\ref{lem:annealed}.

To show \eqref{eq:quenched}, we rewrite
\begin{equation*}
  \EE\Big[E^{\tau}_{\nu}\big[
      e^{-\frac{\lambda}{g_N}S_{\cD}(tR_N)}\big]^2 \Big|\,\cG\Big]
  = \EE\Big[\hat{E}^{\tau}_{\nu}\big[
      e^{-\frac{\lambda}{g_N}\sum_{x\in \cD_N}
        (\ell_{tR_N}^{(1)}(x)+\ell_{tR_N}^{(2)}(x))\tau_x}\big]\Big |
    \,\mathcal G\Big],
\end{equation*}
where $\ell^{(1)}$ and $\ell^{(2)}$ are the local times of two independent
Markov chains $Y^{(1)}$ and $Y^{(2)}$, both having law $P^{\tau}_{\nu}$, and
$\hat{E}^{\tau}_{\nu}$ is the expectation with respect to the joint law
$\hat{P}^{\tau}_{\nu}$ of these chains. Again $\PP$-a.s.~for $N$ large enough
the separation event $\mathscr{S}$ holds, and on this event
the law $\hat{P}^{\tau}_{\nu}$ is $\mathcal G$-measurable. Therefore
we can exchange the expectations similarly as before. As in \lemref{lem:annealed}, it
will be enough to control the expression on an event of
$\hat{P}^{\tau}_{\nu}$-probability tending to $1$ as $N\to\infty$. We thus
set
$\hat{\cA}=\cA^{(1)}\cap\cA^{(2)}$ where $\cA^{(i)}$
are defined for both chains $Y^{(i)}$ as in \eqref{eq:eventA}.
Applying \propref{prop:conclocal} and
\lemref{lem:nolargelocaltimes} for both
independent chains, we have that $\PP$-a.s.~as $N\to\infty$,
$\hat{P}^{\tau}_{\nu}[\hat{\cA}^c]=O(e^{-cN})$.

Let $\mathcal C$ be the event that $Y^{(1)}$ and $Y^{(2)}$ visit
disjoint sets of deep traps,
\begin{equation*}
  \cC=\left\{\{x\in \cD_N:~\ell_{tR_N}^{(1)}(x)>0\}
    \cap\{x\in \cD_N:~\ell_{tR_N}^{(2)}(x)>0\}=\emptyset\right\}.
\end{equation*}
We claim that $\hat{P}^{\tau}_{\nu}[\cC^c]= O(e^{-cN})$, $\PP$-a.s.~as
$N\to\infty$. Indeed, by
\lemref{lem:nolargelocaltimes},
with probability larger than $1-c2^{-\frac{1}{4}\varepsilon_0 N}$, the chain
$Y^{(1)}$ visits at most $2^{\frac{3}{2}\varepsilon_0N}$ different vertices in
$\cD_N$. By \lemref{lem:approxhit}, each of those vertices has probability
smaller than $c2^{(\gamma-1+\varepsilon)N}$ of being hit by $Y^{(2)}$, for every
$\varepsilon>0$, $\PP$-a.s.~for $N$ large enough. Therefore by the choice
\eqref{eq:epsilon0} of $\varepsilon_0$, $\PP$-a.s.~for $N$ large enough,
\begin{equation*}
	\hat{P}^{\tau}_{\nu}[\cC^c] \leq c2^{-\frac{1}{4}\varepsilon_0 N}
		+ 2^{\frac{3}{2}\varepsilon_0 N} c'2^{(\gamma-1+\varepsilon)N}
		\leq c2^{-\frac{1}{4}\varepsilon_0 N} + c'2^{-\frac{1}{2}\varepsilon_0N
			+\varepsilon N},
\end{equation*}
which decays exponentially if $\varepsilon<\varepsilon_0/2$.

Since on $\cC$ the $\tau_x$ of the vertices $x\in \cD_N$ visited by
$Y^{(1)}$ and $Y^{(2)}$ are independent, and since the integrand is bounded
by 1, we have on the separation event $\mathscr{S}$,
\begin{align*}
	\EE&\left[E^{\tau}_{\nu}\left[
			e^{-\frac{\lambda}{g_N}S_{\cD}(tR_N)}\right]^2 \mid\cG\right]\\
	&=\hat{E}^{\tau}_{\nu}\left[\overline{\EE}\left[
		e^{-\frac{\lambda}{g_N}\sum_{x\in \cD_N}
		(\ell_{tR_N}^{(1)}(x)+\ell_{tR_N}^{(2)}(x))\tau_x}\right]\right]\\
	&=\hat{E}^{\tau}_{\nu}\left[\overline{\EE}\left[
		e^{-\frac{\lambda}{g_N}\sum_{x\in \cD_N}
		(\ell_{tR_N}^{(1)}(x)+\ell_{tR_N}^{(2)}(x))\tau_x}\right]
		\mathbf{1}_{\hat{\cA}\cap\cC} \right]	+ O(e^{-cN})\\
	&= \hat{E}^{\tau}_{\nu}\left[\overline{\EE}\left[
		e^{-\frac{\lambda}{g_N}\sum_{x\in \cD_N}
		\ell_{tR_N}^{(1)}(x)\tau_x}\right]
		\overline{\EE}\left[
		e^{-\frac{\lambda}{g_N}\sum_{x\in \cD_N}
		\ell_{tR_N}^{(2)}(x)\tau_x}\right]
		\mathbf{1}_{\hat{\cA}\cap\cC} \right]	+ O(e^{-cN}).
\end{align*}
Using the same procedure as in the proof of \lemref{lem:annealed}, on
the event $\hat{\cA}$, the two inner expectations,
$x\in\cD_N$, both converge to
\begin{equation*}
	\exp\bigg\{-\cK \lambda^{\alpha} 2^{(\gamma'-\gamma)N}
		\sum_{x\in \cD_N} \ell_{tR_N}^{(i)}(x)^{\alpha}\bigg\}
	=\exp\big\{-\cK\lambda^{\alpha}L_N^{(i)}(t)\big\}, \quad i=1,2.
\end{equation*}
Moreover, on $\hat{\cA}$, the
local time functionals $L_N^{(i)}(t)$ concentrate on $t$ simultaneously.
It follows that on $\mathscr{S}$, $\PP$-a.s.~as $N\to\infty$,
\begin{equation*}
  \EE\left[E^{\tau}_{\nu}\left[
      e^{-\frac{\lambda}{g_N}S_{\cD}(tR_N)}\right]^2 \Big |\, \cG \right]
  = e^{-2\cK \lambda^{\alpha} t} +o(1).
\end{equation*}
Noting again that $\mathscr{S}$ occurs $\PP$-a.s.~for $N$ large enough,
this shows \eqref{eq:quenched}, and hence the lemma.
\end{proof}

\begin{remark}\label{rem:convmode}
  (a) Inspecting the last proof carefully, it follows that
  Lemma~\ref{lem:onedim} can be slightly strengthened. Namely, the stated
  convergence holds a.s.~with respect to $\xi$ and
  $\underline{E}$, and in probability only with respect to
  $\overline{E}$. The same remark then applies to Theorem~\ref{thm:main}.

  (b) A closer analysis of the errors made in the computation of the quasi-annealed
  Laplace transform, in particular in \eqref{eq:squarexp}, shows that the error
  in \lemref{lem:annealed} and \eqref{eq:quenched} is of order
  $O(N^{-1}\log^2N)$, where the logarithmic part comes from the
  $\log^2 u$ part in \eqref{eq:squarexp}, $u$ being bounded by $N^2m_N$, and
  $m_N$ being polynomial in $N$.
  Therefore the variance decay is not enough to apply the
  Borel-Cantelli lemma and obtain $\PP$-a.s.~convergence.

  (c) Note also that the previous proof, more precisely bounding the
  $\log^2 u$ part of \eqref{eq:squarexp}, requires that
  $\log (N^2 m_N)\ll N^{1/2}$. This is where our improved techniques to
  estimate the spectral gap in \propref{prop:spectralgap} are necessary.
  As we already remarked, the techniques of \cite{FIKP98} show roughly
  that $m_N \le e^{\sqrt {N \log N}}$ only, which is not sufficient.
\end{remark}


\subsection{Joint convergence of increments}

In the next step, we extend the convergence to joint convergence of increments.

\begin{lemma} \label{lem:increments}
The increments of the rescaled clock processes $g_N^{-1}S_{\cD}(tR_N)$
converge jointly in $\PP$-probability in $P^{\tau}_{\nu}$-distribution to the
increments of an $\alpha$-stable subordinator.
\end{lemma}

\begin{proof}
Fix $k\geq1$ and $0=t_0<t_1<\cdots<t_k $. We will show that for every
$\lambda_1,\dots,\lambda_k \in (0,\infty)$ and $\PP$-a.e.~environment
$\tau $,
\begin{equation}
  \label{eq:increments}
  \lim_{N\to\infty}E^{\tau}_{\nu}\left[e^{-\frac{1}{g_N} \sum_{i=1}^k
      \lambda_i(S_{\cD}(t_iR_N)-S_{\cD}(t_{i-1}R_N))} \right]
  = \lim_{N\to\infty} \prod_{i=1}^k E^{\tau}_{\nu}\left[e^{-\frac{\lambda_i}{g_N}
      S_{\cD}((t_i-t _{i-1})R_N)} \right].
\end{equation}
Then the lemma follows by using the above proved convergence in
$\PP$-probability in $P^{\tau}_{\nu}$-distribution of the one-dimensional
marginals.

Let $ I^i=[t_{i}R_N-N^2m_N,t_iR_N]$. For a set $I\subset[0,\infty)$, let
$\mathcal V(I)$ be the event
\begin{equation*}
	\cV(I) = \{Y_s\notin \cD_N \text{ for all }s\in I\}.
\end{equation*}
On the event $\cV\left(\cup_{i=1}^k I^i\right)$, for every $i\le k$,
\begin{equation}\label{eq:eventv}
	S_{\cD}(t_iR_N)-S_{\cD}(t_{i-1}R_N)
		= S_{\cD}(t_iR_N-N^2m_N) - S_{\cD}(t_{i-1}R_N).
\end{equation}
Moreover, by \lemref{lem:approxhit},  $\PP$-a.s.~for all $x\in \cD_N$, for
$\varepsilon>0$ small and $N$ large enough,
\begin{equation*}
  P^{\tau}_{\nu}[H_x \leq N^2m_N] \leq 2^{(\varepsilon-1)N}.
\end{equation*}
By \eqref{eq:sizetop}, $|\cD_N|\leq c2^{(1-\gamma')N}$, hence the expected
number of vertices $x\in \cD_N$ visited in a time-interval of length $N^2m_N$
is smaller than $c2^{(\varepsilon-\gamma')N}$, $\PP$-a.s.~for $N$ large enough.
This still holds for a finite union of intervals of length $N^2 m_N$, and so
we conclude that by the Markov inequality,
$P^{\tau}_{\nu}\left[\cV\left(\cup_{i=1}^k I^i\right)\right]\to1$,
$\PP$-a.s.~as $N\to\infty$.

The reason to shorten the time intervals as above is to give the Markov chain
$Y$ the time it needs to mix. Define the event
\begin{equation*}
  \cM = \{\tx\circ\theta_{t_iR_N-N^2m_N}\leq N^2m_N\,\forall i=1,\dots,k\}.
\end{equation*}
It is easy to see using Proposition~\ref{prop:mixtime} that
 $P^{\tau}_{\nu}[\cM] \to 1$, $\PP$-a.s.~as $N\to\infty$. On
the event $\cM$ the Markov chain $Y$ always mixes between $t_iR_N-N^2m_N$
and $t_i R_N$ and thus, by \lemref{l:tmixprop}, for every $i=1,\dots,k$
and $y\in\HH_N$,
\begin{equation*}
  P^{\tau}_\nu [Y_{t_iR_N} = y\mid \cM]=\nu_y.
\end{equation*}
Therefore, on $\cM$,
\begin{equation} \label{eq:eventm}
	\left(S_{\cD}(t_iR_N-N^2m_N)-S_{\cD}(t_{i-1}R_N)\right)_{i=1,\dots,k}
	\stackrel{d}{=} \left( S_{\cD}^{(i)}((t_i-t _{i-1})R_N-N^2m_N)
	\right)_{i=1,\dots,k},
\end{equation}
where the $S_{\cD}^{(i)}$ are the clock processes of the deep traps of
independent stationary started processes $Y^{(i)}$ having the same law as $Y$.

Combining observations \eqref{eq:eventv} and \eqref{eq:eventm}, with
the estimates on the probabilities of
$\cV\left(\cup_{i=1}^k I^i\right)$ and $\cM$, since the integrand is
bounded by 1, we obtain that $\PP$-a.s.~as $N\to\infty$,
\begin{align*}
	E^{\tau}_{\nu}&\left[e^{-\frac{1}{g_N} \sum_{i=1}^k
			\lambda_i(S_{\cD}(t_iR_N)-S_{\cD}(t_{i-1}R_N))} \right]\\
	&= E^{\tau}_{\nu}\left[e^{-\frac{1}{g_N} \sum_{i=1}^k
			\lambda_i(S_{\cD}(t_iR_N-N^2m_N)-S_{\cD}(t_{i-1}R_N))}
			\mathbf{1}_{\cV\left(\cup_{i=1}^k I^i\right) \cap \cM} \right] +o(1)\\
	&= E^{\tau}_{\nu}\left[\prod_{i=1}^k E^{\tau}_{\nu}\left[
			e^{-\frac{\lambda_i}{g_N} S_{\cD}^{(i)}((t_i-t _{i-1})R_N-N^2m_N)}\right]
			\mathbf{1}_{\cV\left(\cup_{i=1}^k I^i\right) \cap \cM} \right] +o(1)\\
	&= \prod_{i=1}^k E^{\tau}_{\nu}\left[
			e^{-\frac{\lambda_i}{g_N} S_{\cD}^{(i)}((t_i-t _{i-1})R_N-N^2m_N)}
			\right] +o(1).
\end{align*}
Using analogous arguments it can be shown that
for every $i=1,\dots,k$, $\PP$-a.s.~as $N\to\infty$,
\begin{align*}
	E^{\tau}_{\nu}\left[
			e^{-\frac{\lambda_i}{g_N} S_{\cD}^{(i)}((t_i-t _{i-1})R_N-N^2m_N)}
			\right]
	= E^{\tau}_{\nu}\left[e^{-\frac{\lambda_i}{g_N}
			S_{\cD}^{(i)}((t_i-t _{i-1})R_N)}\right] +o(1).
\end{align*}
Combining the last two equations proves \eqref{eq:increments} and hence the lemma.
\end{proof}


\subsection{Tightness in the Skorohod topology}

The last step in the proof of \propref{prop:deep} is to show tightness.

\begin{lemma} \label{lem:tight}
The sequence of probability measures
$P^{\tau}_{\nu}\big[g_N^{-1}S_{\cD}(tR_N)\in\,\cdot\,\big]$ is
$\PP$-a.s.~tight with respect to the Skorohod $M_1$-topology on $D([0,T],\RR)$.
\end{lemma}

\begin{proof}
  The proof is standard but we include it for the sake
  of completeness. By \cite[Theorem~12.12.3]{W02}, the tightness in the
  Skorohod $M_1$-topology on $D([0,T],\RR)$ is characterized in the
  following way: For $f\in D([0,T],\RR)$, $\delta>0$, $t\in[0,T]$, let
  \begin{align*}
    w_f(\delta) &= \sup\left\{\inf_{\alpha\in[0,1]}|f(t)-(\alpha f(t_1)
        +(1-\alpha)f(t_2))|:~t_1\leq t\leq t_2\leq T,~t_2-t_1
      \leq \delta\right\},\\
    v_f(t,\delta) &= \sup\left\{|f(t_1)-f(t_2)|:~t_1,t_2\in[0,T]\cap
      (t-\delta,t+\delta)\right\}.
  \end{align*}
  The sequence of probability measures
  $P_N=P^{\tau}_{\nu}\big[g_N^{-1}S_{\cD}(tR_N)\in\,\cdot\,\big]$ on
  $D([0,T],\RR)$ is tight in the $M_1$-topology, if
  \begin{enumerate}[(i)]
    \item For every $\varepsilon>0$ there is $c$ such that
    \begin{equation}\label{eq:tight1}
      P_N[f:~\|f\|_{\infty}>c]\leq \varepsilon, \quad N\geq1.
    \end{equation}
    \item For every $\varepsilon>0$ and $\eta>0$, there exist $\delta\in(0,T)$
    and $N_0$ such that
    \begin{equation}\label{eq:tight2}
      P_N[f:~w_f(\delta)\geq\eta]\leq\varepsilon, \quad N\geq N_0,
    \end{equation}
    and
    \begin{equation} \label{eq:tight3}
      P_N[f:~v_f(0,\delta)\geq\eta]\leq\varepsilon \text{ and }
      P_N[f:~v_f(T,\delta)\geq\eta]\leq\varepsilon, \quad N\geq N_0.
    \end{equation}
  \end{enumerate}

Since the clock processes are increasing, \eqref{eq:tight1} is equivalent
to convergence of the distribution of $g_N^{-1}S_{\cD}(TR_N)$, which
follows from the convergence of the Laplace transform of the marginal at time
$T$. \eqref{eq:tight2} is immediate from the fact that the oscillating
function $w_f(\delta)$ is always zero since the processes
$g_N^{-1}S_{\cD}(tR_N)$ are increasing. To check \eqref{eq:tight3},
again by the monotonicity of the $g_N^{-1}S_{\cD}(tR_N)$ it is enough to
check that for $\delta$ small enough and $N\geq N_0$,
$P^{\tau}_{\nu}[g_N^{-1}S_{\cD}(\delta R_N)\geq\eta]\leq \varepsilon$. By the
convergence of the marginal at time $\delta$, we may take $\delta$ such that
$\PP[V_{\alpha}(\delta)\geq\eta]\leq \frac{\varepsilon}{2}$ and $N_0$ such that
for $N\geq N_0$,
\begin{equation*}
	\left|P^{\tau}_{\nu}\left[\frac{1}{g_N}S_{\cD}(\delta R_N)\geq\eta\right]
		-\PP\left[V_{\alpha}(\delta)\geq\eta\right]\right| \leq \frac{\varepsilon}{2}.
\end{equation*}
The reasoning for $v_f(T,\delta)$ is similar.
\end{proof}


\section{Shallow traps}
\label{sec:shallow}

In this section we show that the convergence of the clock process of the deep
traps shown in \secref{sec:deep} is enough for convergence of the clock
process itself.

\begin{proposition} \label{prop:shallow}
Under the assumptions of \thmref{thm:main}, the clock process of the deep
traps approximates the clock process, namely, for every $t\geq0$,
\begin{equation*}
  \frac{1}{g_N}\big(S(tR_N)-S_{\cD}(tR_N)\big)
  \xrightarrow{N\to\infty} 0 \qquad
	\PP\text{-a.s.~in $P^{\tau}_{\nu}$-probability}.
\end{equation*}
\end{proposition}

\begin{proof}
  We will split the set of shallow traps $\cS_N:=\mathbb H_N\setminus \cD_N$
  into two parts and separately deal with the corresponding contributions
  to the clock process.

  We start with `very shallow traps'. Let
  $\delta>0$ be a small constant which will be fixed later and
  $h_N= e^{\delta\alpha\beta^2 N}$. Define the set of very shallow
  traps as
  \begin{equation*}
    \overline{\cS}_N = \{x\in\HH_N:~\tau_x \leq h_N\}.
  \end{equation*}

  The contribution of this set to the clock process can easily be
  neglected as follows. Write
  \begin{equation*}
    \begin{split}
      E^{\tau}_{\nu}\Bigg[\frac{1}{g_N}\int_0^{tR_N}(1\vee\tau_{Y_s})
        \Ind{Y_s\in \overline{\cS}_N}ds\Bigg]
      &=\frac{1}{g_N} \sum_{x\in\overline{\cS}_N} (1\vee\tau_x)
      E_{\nu}^{\tau}\big[\ell_{tR_N}(x)\big]
    \end{split}
  \end{equation*}
  Note that
  $E^{\tau}_{\nu}[\ell_{tR_N}(x)] = \nu_x tR_N = Z_N^{-1}(1\wedge\tau_x)tR_N$,
  and $(1\vee\tau_x)(1\wedge\tau_x)=\tau_x\leq h_N$ on $\overline{\cS}_N$.
  With \eqref{e:boundzN} for $Z_N$, and \lemref{lem:boundrn} for $R_N$,
  for every $\epsilon>0$, $\PP$-a.s.~for $N$ large enough, the right-hand
  side of the last equation can be bounded from above by
  \begin{equation*}
    g_N^{-1} 2^N h_N Z_N^{-1} tR_N
    \leq c g_N^{-1} e^{\delta\alpha\beta^2 N} 2^{(\gamma+\epsilon)N}.
  \end{equation*}
  To obtain exponential decay of this expression, it is enough to take account
  of the exponential part of $g_N$, which is $e^{\alpha\beta^2N}$. Then, up to
  sub-exponential factors, using that
  $\gamma=\frac{\alpha^2\beta^2}{2\log2}$, the above is bounded by
  \begin{equation*}
    \exp\big\{((\delta-1)\alpha\beta^2 + \frac{1}{2}\alpha^2\beta^2
        + \epsilon\log2) N\big\}.
  \end{equation*}
  Since $\alpha<1$, by choosing $\epsilon$ and $\delta$ small enough this can
  be made smaller than $e^{-cN}$ for some $c>0$. Applying the Markov
  inequality and  the Borel-Cantelli lemma,
  \begin{equation} \label{eq:veryshallow}
    \frac{1}{g_N}\int_0^{tR_N}(1\vee\tau_{Y_s})\Ind{Y_s\in \overline{\cS}_N}ds
    \xrightarrow 0 \qquad \PP\text{-a.s.~in $P_{\nu}^{\tau}$-probability.}
  \end{equation}

  To control the contribution of the remaining shallow traps
  $\mathcal S_N\setminus \overline {\mathcal S}_N$, we first split this
  set into slices $\cS^i_N$ as follows. Set
  \begin{equation*}
    I_N = \ceil{\frac{1}{\log2}(\log g'_N - \log h_N)}.
  \end{equation*}
  Note that by definition of $g'_N$ and $h_N$, for $\delta$ small
  as fixed above, $I_N = cN + O(1)$ for some $c>0$.
  For $i=1,\dots,I_N$, let
  \begin{equation*}
    \cS^i_N=\big\{x\in\mathcal S_N\setminus\overline{\mathcal
        S}_N:~\tau_x \in [2^{-i}g'_N,2^{-i+1}g'_N)\big\},
  \end{equation*}
  so that
  $\mathcal S_N\setminus\overline{\mathcal S}_N= \cup_{i=1}^{I_N} \mathcal S_N^i$.

  We next control the sizes of the slices $\cS^i_N$.
  By the tail approximation \eqref{eq:gaussapprox},
  for all $i=1,\dots,I_N$,
  \begin{equation}\label{eq:probslice}\begin{split}
      \PP[y\in \cS^i_N]
      &\leq \PP\Big[E_x > \frac{1}{\beta\sqrt{N}}(\log g'_N - i \log2) \Big]\\
      &= f_{N,i}^{(1)}	\exp\Big\{-\frac{1}{2} \alpha'^2\beta^2N
        + \alpha'i\log2 - f_{N,i}^{(2)} - o(1)\Big\}(1+o(1)).
  \end{split}\end{equation}
  We separately control the two expressions $f_{N,i}^{(1)}$ and
  $f_{N,i}^{(2)}$. The first one equals
  \begin{equation*}
    f_{N,i}^{(1)} = \frac{\alpha'\beta\sqrt{2\pi N}}
    {\frac{\sqrt{2\pi}}{\beta\sqrt{N}}(\log g'_N - i \log2)}.
  \end{equation*}
  To control this, note that by definition of $I_N$,
  for all $i=1,\dots,I_N$,
  \begin{equation*}
    \log g'_N - i \log2 \geq \log h_N - \log2 = \delta\alpha\beta^2N -\log2.
  \end{equation*}
  It follows that, for all $i=1,\dots,I_N$, $f_{N,i}^{(1)}$ is bounded
  by some constant $c>0$, which can be chosen to be independent of $i$.
  The second expression to control in \eqref{eq:probslice} is
  \begin{equation*}
    f_{N,i}^{(2)} = \frac{i^2 \log^2 2}{2\beta^2 N}
    + \frac{i\log2}{\alpha'\beta^2 N}\log(\alpha'\beta\sqrt{2\pi N}).
  \end{equation*}
  This is strictly positive, so it can be omitted in \eqref{eq:probslice} in
  order to obtain an upper bound.
  Using the obtained control on $f_{N,i}^{(1)}$ and $f_{N,i}^{(2)}$
  in \eqref{eq:probslice}, as well as the
  fact that $\gamma'=\frac{\alpha'^2\beta^2}{2\log2}$, we conclude that
  for all $i=1,\dots,I_N$,
  \begin{equation*}
    \PP[y\in \cS^i_N] \leq c 2^{-\gamma'N} 2^{\alpha'i}.
  \end{equation*}
  In particular, the size $|\cS^i_N|$ of the $i$-th slice is dominated by a
  binomial random variable with parameters $n=2^N$ and
  $p=c2^{\alpha' i}2^{-\gamma'N}$. Then it follows by the
  Markov inequality that for every $\epsilon>0$,
  \begin{equation*}
    \PP\big[|\cS^i_N|> 2^{\epsilon N} c2^{\alpha' i}2^{(1-\gamma')N}\big]
    \leq 2^{-\epsilon N}.
  \end{equation*}
  Since $I_N = cN+O(1)$, a union bound and the Borel-Cantelli lemma imply that
  for every $\epsilon>0$, $\PP$-a.s.~for $N$ large enough,
  \begin{equation} \label{eq:sizeslice}
    |\cS_N^i|\leq 2^{\epsilon_N} c2^{\alpha' i} 2^{(1-\gamma')N},
    \text{ for all }i=1,\dots,I_N.
  \end{equation}

  Coming back to the contribution of the intermediate traps
  $\mathcal S_N\setminus \overline {\mathcal S}_N$ to the clock process,
  we use as before that
  $E^{\tau}_{\nu}[\ell_{tR_N}(y)] = \nu_y tR_N = \frac{1\wedge\tau_y}{Z_N}tR_N$,
  and $(1\vee\tau_y)(1\wedge\tau_y)=\tau_y\leq 2^{-i+1}g'_N$ on $\cS_N^i$.
  With \eqref{e:boundzN} for $Z_N$, \lemref{lem:boundrn} for $R_N$, and
  \eqref{eq:sizeslice} for the size of $\cS_N^i$, we obtain that for every
  $\varepsilon>0$, $\PP$-a.s.~for $N$ large enough, for all $i=1,\dots,I_N$,
  \begin{align*}
    E^{\tau}_{\nu}\Bigg[\frac{1}{g_N}\int_0^{tR_N}(1\vee\tau_{Y_s})
      \Ind{Y_s\in \cS^i_N}ds\Bigg]
    &= \frac{1}{g_N}\sum_{y\in\cS_N^i} (1\vee \tau_y)
    E^{\tau}_{\nu}[\ell_{tR_N}(y)]\\
    &\leq g_N^{-1}|\cS_N^i| 2^{-i+1}g'_N Z_N^{-1} t R_N\\
    &\leq c \frac{g'_N}{g_N} 2^{(\alpha-1)i} 2^{(\gamma-\gamma'+2\varepsilon)N}.
  \end{align*}
  Summing over $i=1,\dots,I_N$, $\PP$-a.s.~for $N$ large enough,
  \begin{equation} \label{eq:expshallow}
    E^{\tau}_{\nu}\Bigg[\frac{1}{g_N}\int_0^{tR_N}(1\vee\tau_{Y_s})
      \Ind{Y_s\in \bigcup_{i=1}^{I_N}\cS_N^i}ds\Bigg]\leq c' \frac{g'_N}{g_N}
    2^{(\gamma-\gamma'+2\varepsilon)N}.
  \end{equation}

  We claim that the right hand side of \eqref{eq:expshallow} decays
  exponentially in $N$ for $\varepsilon>0$ small enough. To this end, as
  before, it is enough to take account of the exponential parts in both
  $g_N$ and $g'_N$, which contribute to the right hand side of
  \eqref{eq:expshallow} by
  \begin{equation*}
    e^{(\alpha' - \alpha)\beta^2 N}
    = 2^{(\sqrt{\gamma'}-\sqrt{\gamma})\frac{2\beta}{\beta_c} N}.
  \end{equation*}
  Hence, to show the exponential decay on the right hand side of
  \eqref{eq:expshallow}, it is sufficient to prove that we can choose
  $\varepsilon>0$ small enough, such that
  \begin{equation} \label{eq:sqrtgamma}
    (\sqrt{\gamma'}-\sqrt{\gamma})\frac{2\beta}{\beta_c}
    + \gamma-\gamma'+2\varepsilon <0.
  \end{equation}
  With a first order approximation of the concave function $\sqrt{x}$ at
  $\gamma$,
  \begin{equation*}
    \frac{1}{2\sqrt{\gamma}} (\gamma-\gamma')
    < \sqrt{\gamma} - \sqrt{\gamma'}.
  \end{equation*}
  Since,
  $\frac{1}{2\sqrt{\gamma}}= \frac{\beta_c}{2\alpha\beta}>\frac{\beta_c}{2\beta}$
  and $\alpha<1$, this implies
  \begin{equation*}
    \frac{\beta_c}{2\beta} (\gamma-\gamma')
    < \sqrt{\gamma} - \sqrt{\gamma'},
  \end{equation*}
  and  \eqref{eq:sqrtgamma} thus holds for $\varepsilon>0$ small enough.
  The right hand side of \eqref{eq:expshallow} then
  decays exponentially, and with Markov inequality we conclude that
  \begin{equation*}
    \frac{1}{g_N}\int_0^{tR_N}(1\vee\tau_{Y_s})
    \Ind{Y_s\in \bigcup_{i=1}^{I_N}\cS_N^i}ds
    \xrightarrow{N\to\infty} 0 \quad \PP\text{-a.s.~in
		$P_{\nu}^{\tau}$-probability.}
  \end{equation*}
  This together with \eqref{eq:veryshallow} finishes the proof of the
  proposition.
\end{proof}

\section{Conclusion} 
Theorem~\ref{thm:main} is a direct consequence of
Propositions~\ref{prop:deep}, \ref{prop:shallow} and
Lemma~\ref{lem:boundrn}.

\bigskip
\bigskip


\appendix 
\section{Extremal characterization of mean hitting time}\label{appendix}

In this appendix we give the proof of the formula \eqref{eq:boundbycond}
which gives a lower bound on the mean hitting time of a set when starting
from stationarity. This formula is a
continuous-time version of (a half of) Proposition~3.2 from \cite{CTW11}. This
proposition, as well
as the underlying result \cite[Proposition~3.41]{AFb}, are stated for a
continuous-time Markov chain whose waiting times are mean-one exponential
random variables.
We were not able to find analogous statements for general continuous-time
Markov chains in the literature, so we provide short proofs here,
for the sake of completeness.

We start by introducing some notation.
Let $Y$ be a reversible continuous-time Markov chain on a finite state space $\cS$
with transition rates $q_{xy}$ and invariant probability measure $\nu_x$, denote by
$P_{\nu}$ and $P_x$ the laws of $Y$ started stationary and from $x$
respectively, and by $E_{\nu}$, $E_x$ the corresponding expectations.
Define the conductances as $c_{xy}=\nu_x q_{xy} = \nu_y q_{yx}$. Let
$q_x=\sum_y q_{xy}$ and $c_x = \sum_y c_{xy}$. The transition
probability from $x$ to $y$ is
$p_{xy}=\frac{q_{xy}}{q_x}=\frac{c_{xy}}{c_x}$.
In the same way as in \secref{sec:defs}, we define the hitting time $H_x$ and the
return time $H^+_x$ to $x$ by $Y$, and similarly $H_A$ and $H^+_A$ for sets
$A\subset\cS$.

A function $g$ on $\cS$ is called harmonic in $x$, if
$\sum_{y}g(y)p_{xy} = g(x)$. For $x\in\cS$ and $B\subset\cS\setminus\{x\}$, the
equilibrium potential $g^{\star}_{x,B}$ is defined as the unique function on
$\cS$ that is harmonic on $(x\cup B)^c$, 1 on $x$ and $0$ on $B$. It is
well known that
\begin{equation*}
  g^{\star}_{x,B}(y) = P_y[H_x \leq H_B].
\end{equation*}
For a function $g:\cS\to\RR$, the Dirichlet form is defined as
\begin{equation}\label{def:dirichlet}
  D(g,g) = \frac{1}{2} \sum_{z\in\cS} \sum_{y\sim z} \nu_z q_{zy}
  (g(z)-g(y))^2,
\end{equation}
where $y\sim z$ means that $y$ and $z$ are neighbors in the sense that
$q_{zy}>0$.

The following proposition is the required generalization of
Proposition~3.2 of \cite{CTW11}.

\begin{proposition}
  \label{prop:appendix}
  For every $x\in \mathcal S$ and $B\subset \mathcal S\setminus\{x\}$
  \begin{equation}
    \label{e:appendix}
    \frac{1}{E_{\nu}[H_x]} \leq D(g^{\star}_{x,B},g^{\star}_{x,B})
    \nu(B)^{-2} = c_x P_x[H^+_x> H_B]\nu (B)^{-2}.
  \end{equation}
\end{proposition}

To prove this proposition we will need a lemma which is a
generalization of  \cite[Proposition~3.41]{AFb} giving the extremal
characterization of the mean hitting time.
\begin{lemma}
  \label{l:extremalhx}
  For every $x\in \mathcal S$,
  \begin{equation}
  \label{e:extremalhx}
    \frac{1}{E_{\nu}[H_x]} = \inf\left\{ D(g,g):~g:\cS\to\RR,~g(x)=1, \sum_{y
        \in\cS}\nu_y g(y)=0\right\}.
  \end{equation}
\end{lemma}

\begin{proof}
  The proof follows the lines of
  \cite{AFb} with some minor changes to fit into the
  setting of general continuous-time chains.

  We first show that there is a minimizing function $g$ that equals
  $g(y) = \frac{Z_{yx}}{Z_{xx}}$, where
  \begin{equation*}
    Z_{yx} = \int_0^{\infty} \Big(P_y[Y_t=x]-\nu_x\Big) \,dt.
  \end{equation*}

  To this end, we introduce the Lagrange multiplier $\gamma$ and consider $g$
  as the minimizer of $D(g,g) + \gamma\sum_z \nu_z g(z)$ with $g(x)=1$. The
  contribution to this of $g(y)$ for $y\neq x$ is
  \begin{equation*}
    \sum_{z\sim y} \nu_y q_{yz} (g(y)-g(z))^2 + \gamma \nu_y g(y),
  \end{equation*}
  which is minimized if
  \begin{equation*}
    2\sum_{z\sim y} \nu_y q_{yz} (g(y)-g(z)) + \gamma \nu_y =0.
  \end{equation*}
  From this we get for all $y\in\cS$, by introducing the term including the
  parameter $\beta$ for the case $y=x$, that
  \begin{equation*}
    g(y) = \sum_{z\sim y} \frac{q_{yz}}{q_y} g(z) - \frac{\gamma}{2}
    \frac{1}{q_y} + \frac{\beta}{q_y} \Ind{y=x}.
  \end{equation*}
  Multiplying by $q_y$ and $\nu_y$, and summing over all $y\in\cS$,
  \begin{equation*}
    \sum_{y} \sum_{z\sim y} \nu_y q_{yz} g(y) = \sum_y \sum_{z\sim y} \nu_y
    q_{yz} g(z) - \frac{\gamma}{2} + \beta \nu_x.
  \end{equation*}
  By reversibility $\nu_y q_{yz} = \nu_z q_{zy}$, so the term on the left and
  the first term on the right are identical, which gives $\frac{\gamma}{2} =
  \beta \nu_x$. Thus there is a minimizing $g$ such that
  \begin{equation}
    \label{e:gprop}
    g(y) =   \frac{\beta}{q_y}
    \big(\Ind{y=x}-\nu_x\big) +
    \sum_{z\sim y} \frac{q_{yz}}{q_y} g(z).
  \end{equation}

  We now show that up to the factor $\beta$ the function $y\mapsto Z_{yx}$ satisfies
  the same relation. Indeed, by the strong Markov property at the time $J_1$
  of the first jump of $Y$, which under $P_y$ is an exponential random variable
  with mean $\frac{1}{q_y}$,
  \begin{equation*}
    \begin{split}
      Z_{yx} &= \int_0^{\infty}\left(\int_0^{J_1}
        \big(\Ind{y=x}-\nu_x\big) dt
        +\sum_{z\sim y} \frac{q_{yz}}{q_y} \int_0^{\infty}
        \big(P_z[Y_t=x]
        -\nu_x\big) dt\right)dP_y(J_1) \\
      &= \frac{1}{q_y}\big(\Ind{y=x}-\nu_x\big)  + \sum_{z\sim y}
      \frac{q_{yz}}{q_y} Z_{zx}.
    \end{split}
  \end{equation*}
  The function $g(y)= \frac{Z_{yx}}{Z_{xx}}$ thus satisfies the
  constrains of the variational problem in \eqref{e:extremalhx} and
  fulfills \eqref{e:gprop} with $\beta = 1/Z_{xx}$. It is thus the
  minimizer of this variational problem.

  Moreover, by \cite[Lemmas~2.11~and~2.12]{AFb}, we have
  $Z_{xx}=E_{\nu}[H_x] \nu_x$ and $\nu_x E_y[H_x] = Z_{xx}-Z_{yx}$.
  Denoting $h(y)=E_y[H_x]$ and using these equalities,
  we obtain
  \begin{equation*}
    D(g,g) = \frac{1}{E_{\nu}[H_x]^2}D(h,h) = \frac{1}{E_\nu [H_x]},
  \end{equation*}
  where for the last equality we used  $D(h,h) = E_{\nu}[H_x]$, by
  e.g.~\cite[Lemma~6]{AB92}. This completes the proof.
\end{proof}

With this lemma the proof of Proposition~\ref{prop:appendix} follows the
lines of \cite{CTW11}.

\begin{proof}[Proof of Proposition~\ref{prop:appendix}]
  To prove the inequality in \eqref{e:appendix}, it is sufficient to
  modify the function $g^{\star}_{x,B}$ so that it becomes admissible
  for the variational problem in \lemref{l:extremalhx}. Write
  $g^{\star}$ for $g^{\star}_{x,B}$ and define $\tilde{g}$ on $\cS$ as
  \begin{equation*}
    \tilde{g} (z)=\frac{g^{\star}(z)-\sum_{y\in\cS}\nu_y g^{\star}(y)}
    {1-\sum_{y\in\cS}\nu_y g^{\star}(y)}.
  \end{equation*}
  Then $\tilde{g}$ equals $1$ on $x$ and
  $\sum_{z\in\cS}\nu_z \tilde{g}(z)=0$.
  Hence, by \lemref{l:extremalhx},
  \begin{equation*}
    \frac{1}{E_{\nu}[H_x]} \leq D(\tilde{g},\tilde{g})
    = D(g^{\star},g^{\star}) \left(1-\sum_{y\in\cS}\nu_y g^{\star}(y)
    \right)^{-2}.
  \end{equation*}
  But $g^{\star}$ is non-negative, bounded by $1$ and non-zero only on $B^c$,
  therefore $\sum_{y\in\cS}\nu_y g^{\star}(y) \leq \nu(B^c)$, the first
  part of \propref{prop:appendix} follows.

  To prove the equality in \eqref{e:appendix}, we show that
  \begin{equation}
    \label{e:dfads}
    D(g^{\star}_{x,B},g^{\star}_{x,B}) = P_x[H^+_x > H_B]c_x.
  \end{equation}
  Indeed, let again $g^{\star}=g^{\star}_{x,B}$. If $g^{\star}$ is harmonic in $z$, the
  second sum in the Dirichlet form \eqref{def:dirichlet} is
  \begin{equation*}
    \sum_{y\sim z} c_{zy}(g^{\star}(z)-g^{\star}(y))^2
    = \sum_{y\sim z} c_{zy} (g^{\star}(y)^2 - g^{\star}(z)^2).
  \end{equation*}
  This shows that the contribution to the Dirichlet form of every edge that
  connects two vertices in which $g^{\star}$ is harmonic or zero vanishes.
  Therefore $D(g^{\star},g^{\star})$ reduces to
  \begin{align*}
    D(g^{\star},g^{\star}) &=\frac{1}{2}\left(
      \sum_{y\sim x} c_{xy}(1-g^{\star}(y))^2
      + \sum_{y\sim x} c_{xy}(1-g^{\star}(y)^2)\right)\\
    &= \sum_{y\sim x} c_{xy}(1-g^{\star}(y))\\
    &= c_x \sum_{y\sim x} p_{xy} P_y[H_x>H_B]\\
    &= c_x P_x[H_x^+ > H_B].
  \end{align*}
  This proves \eqref{e:dfads} and thus the proposition.
\end{proof}


%

\providecommand{\bysame}{\leavevmode\hbox to3em{\hrulefill}\thinspace}
\providecommand\MR{}
\renewcommand\MR[1]{\relax\ifhmode\unskip\spacefactor3000
\space\fi \MRhref{#1}{#1}}
\providecommand\MRhref{}
\renewcommand{\MRhref}[2]%
{\href{http://www.ams.org/mathscinet-getitem?mr=#1}{MR #2}}
\providecommand{\href}[2]{#2}

\end{document}